\documentclass[leqno]{amsart}

\usepackage{amsmath,amssymb,amsthm}
\usepackage{mathrsfs}
\usepackage{braket}
\usepackage{graphicx}
\usepackage{amscd}
\usepackage[all]{xy}
\usepackage{comment}
\usepackage{eucal}

\theoremstyle{definition}
\newtheorem{dfn}{Definition}[section]
\newtheorem{lem}[dfn]{Lemma}
\newtheorem{cor}[dfn]{Corollary}
\newtheorem{prp}[dfn]{Proposition}
\newtheorem{theom}[dfn]{Theorem}
\newtheorem{rem}[dfn]{Remark}

\newtheorem{ex}[dfn]{Example}
\newtheorem{conj}[dfn]{Conjeture}
\newcommand{\mf}{\mathfrak}

\DeclareMathOperator{\ch}{\sf ch}
\DeclareMathOperator{\sch}{\sf sch}
\DeclareMathOperator{\sdim}{\sf sdim}
\DeclareMathOperator{\id}{\sf id}
\DeclareMathOperator{\Ind}{Ind}

\DeclareMathOperator{\spn}{\sf span}
\makeatletter

\@addtoreset{equation}{section}
\makeatother

\title[Modular invariant representations of
the $\mathcal{N}=2$ SCA]
{Modular invariant representations of
the $\mathcal{N}=2$ superconformal algebra}
\author{Ryo SATO}
\address{Graduate School of Mathematical Sciences,
The University of Tokyo 
3-8-1 Komaba Meguro-ku Tokyo,
153-8914, Japan}
\email{rsato@ms.u-tokyo.ac.jp}

\begin{document}

\maketitle

\begin{abstract}
We compute the modular transformation formula
of the characters for a certain family 
of (finitely or uncountably many) simple modules over
 the simple $\mathcal{N}=2$ vertex operator superalgebra of central charge
$c_{p,p'}=3\left(1-\frac{2p'}{p}\right),$ where
$(p,p')$ is a pair of coprime positive integers such that $p\geq2$.
When $p'=1$, the formula coincides with that of the $\mathcal{N}=2$ unitary minimal series
found by F.~Ravanini and S.-K.~Yang.
In addition, we study the properties of 
the corresponding ``modular $S$-matrix'', which is no longer a matrix
if $p'\geq2$.
\end{abstract}


\section{Introduction}

One of the most remarkable features in representation theory of 
vertex operator superalgebras (VOSAs) is the modular invariance property 
of the characters of their modules. 
The property is firstly established by Y.~Zhu in \cite{zhu1996modular}
for rational, $C_{2}$-cofinite vertex operator algebras (VOAs)
under some natural conditions.
See \cite{miyamoto2004modular} and \cite{dong2005modularity}
for the generalization to irrational cases and to super cases, respectively.
We note that all these previous works are based on the $C_{2}$-cofiniteness assumption
which is introduced in \cite{zhu1996modular} and is
deeply related to the finite dimensionality of 
the space of $1$-point functions on torus (see \cite[Definition 5.3]{dong2005modularity}).
See \cite{zhu1996modular}, \cite{miyamoto2004modular}, and \cite{dong2005modularity} for more details.

In the present paper, we construct a ``modular invariant'' family
of simple highest weight modules over
the simple VOSA $L_{c_{p,p'}}$
associated with the $\mathcal{N}=2$ superconformal algebra
of central charge
$$c_{p,p'}:=3\left(1-\frac{2p'}{p}\right).$$
Here $(p,p')$ is a pair of coprime positive integers such that $p\geq2$.
We should note that the simple VOSA $L_{c_{p,p'}}$
is $C_{2}$-cofinite if and only if $p'=1$ (see Corollary \ref{nonC2}). 
When it is not $C_{2}$-cofinite,
the dimension of the space of $1$-point functions on torus is not known to be finite
(cf.~\cite[Theorem 8.1]{dong2005modularity}).
In fact, the space spanned by the character functions of
simple $\frac{1}{2}\mathbb{Z}_{\geq0}$-gradable $L_{c_{p,p'}}$-modules
is not finite dimensional (see Remark \ref{infinite}).
Therefore we explain the precise meaning of the ``modular invariance''
in our case below.

For each pair $(p,p')$ as above, our family of simple $L_{c_{p,p'}}$-modules 
is divided into two classes, \emph{atypical} modules and \emph{typical} modules.
All the atypical modules in this paper
are obtained in either way:
\begin{enumerate}
\item from principal admissible $\widehat{\mf{sl}}_{2|1}$-modules
of level $k=-1+\frac{\,p'}{p}$ (see \cite{bowcock1998characters}, \cite{kac2016representations}, and
Appendix \ref{nonunitaryKW})
by the quantum Becchi-Rouet-Stora-Tyutin (BRST) reduction 
(see \cite{kac2003quantum} and \cite{arakawa2005representation}).
\item from the Kac-Wakimoto admissible highest weight $\widehat{\mf{sl}}_{2}$-modules of level
$k'=-2+\frac{p}{p'}$ (see \cite{KW88}) by the Kazama--Suzuki coset construction
(see \cite{KS89}, \cite{hosono1991lie}, \cite{FSST99}, and \cite[\S 7]{sato2016equivalences}).
\end{enumerate}
Similarly to the atypical modules, the typical ones are also 
obtained in either way:
\begin{enumerate}
\item from typical highest weight $\widehat{\mf{sl}}_{2|1}$-modules 
(in the sense of \cite{gorelik2015characters})
of level $k=-1+\frac{\,p'}{p}$ by the quantum BRST reduction
(see Appendix \ref{reduction} for the details).
\item from simple relaxed highest weight $\widehat{\mf{sl}}_{2}$-modules 
of level $k'=-2+\frac{p}{p'}$ 
(see \cite{creutzig2013modular})
by the Kazama--Suzuki coset construction,
\end{enumerate}
In response to the above two classes of modules, we introduce the following two families: 
\begin{enumerate}
\item[(A)]
characters of atypical modules indexed by 
a certain finite set $\mathscr{S}_{p,p'}$ 
(see Definition \ref{discrete}),
\item[(T)]
characters of typical modules
indexed by $K_{p,p'}\times\mathbb{R}$,
where $K_{p,p'}$ is a finite set which parameterizes
the Belavin--Polyakov--Zamolodchikov (BPZ) minimal series
of central charge $1-\frac{6(p-p')^{2}}{pp'}$.
\end{enumerate}
We note that the former characters can be written in terms of 
the \emph{Appell--Lerch sum} (see \cite{zweg2002mock} and \cite{semikhatov2005higher}).
See \cite{kac2016representations} and Lemma \ref{itoAL}
for the details.

Now we explain the modular transformation properties of the characters.
To describle the whole picture of the modular invariance
in super cases, we need to consider 
the following four types of formal characters:
\begin{enumerate}
\item $\ch^{0,0}(M)\colon$ Neveu--Schwarz character,
\item $\ch^{0,1}(M)\colon$ Neveu--Schwarz supercharacter,
\item $\ch^{1,0}(M)\colon$ Ramond character,
\item $\ch^{1,1}(M)\colon$ Ramond supercharacter.
\end{enumerate}
For $\lambda\in\mathscr{S}_{p,p'}$ and $(\mu,x)\in K_{p,p'}\times\mathbb{R}$,
we denote the character functions of the corresponding highest weight modules by
$\mathbf{A}_{\lambda}^{\varepsilon,\varepsilon'}(\tau,u,t)
=\ch^{\varepsilon,\varepsilon'}(\mathcal{L}_{\lambda})$
and $\mathbf{T}_{\mu,x}^{\varepsilon,\varepsilon'}(\tau,u,t)
=\ch^{\varepsilon,\varepsilon'}(\mathcal{L}_{\mu,x}).$
Here $(\tau,u,t)\in\mathbb{H}\times\mathbb{C}\times\mathbb{C}$ stands for
a certain coordinate of the Cartan subalgebra 
of the $\mathcal{N}=2$ superconformal algebra.
Then our ``modular invariance'' in this paper means
the establishment of the following modular $S$-transformation
\begin{align*}
\mathbf{A}^{\varepsilon,\varepsilon'}_{\lambda}
\left(-\frac{1}{\tau},\frac{u}{\tau},t-\frac{u^{2}}{6\tau}\right)
=&\sum_{\lambda'\in\mathscr{S}_{p,p'}}
S_{\lambda,\lambda'}^{aa,(\varepsilon,\varepsilon')}
\mathbf{A}^{\varepsilon',\varepsilon}_{\lambda'}(\tau,u,t)\\
&\ \ \ +\sum_{\mu''\in K_{p,p'}}
\int_{\mathbb{R}}
\mathrm{d}x''\,S_{\lambda,(\mu'',x'')}^{at,(\varepsilon,\varepsilon')}
\mathbf{T}^{\varepsilon',\varepsilon}_{\mu'',x''}(\tau,u,t),\\
\mathbf{T}_{\mu,x}^{\varepsilon,\varepsilon'}\left(-\frac{1}{\tau},\frac{u}{\tau},
t-\frac{u^{2}}{6\tau}\right)
=&\sum_{\mu'\in K_{p,p'}}
\int_{\mathbb{R}}\mathrm{d}x'\,S_{(\mu,x),(\mu',x')}^{tt,(\varepsilon,\varepsilon')}
\mathbf{T}_{\mu',x'}^{\varepsilon',\varepsilon}(\tau,u,t)
\end{align*}
and the (rather trivial) modular $T$-transformation.
See \S \ref{3.2} and \S \ref{4.4} for the details.

At last we give some remarks on 
relationships between our result and the relevant previous works.
\begin{itemize}
\item
When $p'=1$, 
the set $K_{p,1}$ is empty and the index set $\mathscr{S}_{p,1}$ bijectively corresponds to
the $\mathcal{N}=2$ unitary minimal series of central charge
$$c_{p,1}=3\left(1-\frac{2}{p}\right).$$
Therefore the finite-dimensional space
$$\spn_{\mathbb{C}}\bigl\{\,\mathbf{A}^{\varepsilon,\varepsilon'}_{\lambda}
\,\big|\,\varepsilon,\varepsilon'\in\{0,1\},\,\lambda\in\mathscr{S}_{p,1}\,\bigr\}$$
is ${\sf SL}(2,\mathbb{Z})$-invariant.
Then our result recovers the modular transformation of
the $\mathcal{N}=2$ minimal unitary characters which is obtained by
F.~Ravanini and S.-K.~Yang in \cite{ravanini1987modular} (see also \cite{qiu1987modular})
and proved by V.G.~Kac and M.~Wakimoto in \cite{kac1994integrable}.
See Appendix \ref{KWmodular} for the details.
\item In \cite{semikhatov2005higher}, A.~M.~Semikhatov, A.~Taormina, and I.~Yu.~Tipunin 
studied the modular property of
the characters of simple highest weight modules
over the $\widehat{\mf{sl}}_{2|1}$ of level $k=-1+\frac{\,p'}{p}$ for $p'\geq2$.
The $\mathcal{N}=2$ VOSA is obtained from the affine VOSA associated with 
$\widehat{\mf{sl}}_{2|1}$ by the quantum BRST reduction 
(based on the result of T.~Arakawa in \cite{arakawa2005representation})
and the corresponding central charge is given by
$$c_{p,p'}=-3(2k+1).$$
We can verify that the ``admissible $\widehat{\mf{sl}}_{2|1}$ representations'' 
considered in \cite[\S B.3]{semikhatov2005higher} correspond to
principal admissible $\widehat{\mf{sl}}_{2|1}$-modules with spectral flow twists
(cf.~\cite[\S 2]{kac2016representations}).
We note that the character formula for principal admissible 
$\widehat{\mf{sl}}_{2|1}$-modules is proved in
\cite{gorelik2015characters} (see \cite[Lemma 2.1]{kac2016representations}
for the details).

We should mention that 
the reduced version of the formula in \cite[Theorem 4.1]{semikhatov2005higher}
is presented in \cite[(4.2.75)]{ghom2003high}
and our result gives an alternative expression of \cite[(4.2.75)]{ghom2003high} purely in terms 
of the character functions.
\item
By the Kazama--Suzuki coset construction (see \cite{KS89}, \cite{hosono1991lie},
\cite{FST98}, and \cite{sato2016equivalences}),
the modular invariant family of $L_{c_{p,p'}}$-modules in this paper 
can be regarded as the counterpart of 
that of $\widehat{\mf{sl}}_{2}$-modules at the Kac-Wakimoto admissible levels
$k'=-2+\frac{p}{p'}$ studied in \cite{creutzig2013modular}.
It is worth noting that the formal characters 
of ``typical'' $\widehat{\mf{sl}}_{2}$-modules
which are parameterized by $K_{p,p'}\times\mathbb{R}$
are not convergent to functions defined in the upper half plane $\mathbb{H}$.
\end{itemize}

\medskip
{\bf Acknowledgments: }
The author would like to express his gratitude to 
Thomas Creutzig for helpful discussions and valuable comments.
He also would like to express his appreciation to Minoru Wakimoto
for valuable comments on Appendix \ref{nonunitaryKW}, and
to Yoshiyuki Koga for helpful discussions on Appendix \ref{taffinesuper}.
He wishes to thank
Kazuya Kawasetsu, Hisayoshi Matsumoto, Hironori Oya, and Yoshihisa Saito
for fruitful discussions and comments,
We also thank Victor Kac, Antun Milas, and Vladimir Dobrev for letting him
know the references.
Some part of this work is done while he was visiting
Academia Sinica, Taiwan, in February--March 2017.
He is grateful to the institute for its hospitality.
This work is supported by the Program for Leading Graduate
Schools, MEXT, Japan.

\section{Preliminaries}

\subsection{Notation}

For $(\tau,u)\in\mathbb{H}\times\mathbb{C}$, we set $(q,z):=(e^{2\pi i\tau},e^{2\pi iu})$.
We denote the eta function by
$$\eta(\tau):=
q^{\frac{1}{24}}\prod_{n>0}(1-q^{n})$$
 and the theta functions by
$$\vartheta_{\varepsilon,\varepsilon'}(u;\tau):=
z^{\frac{\varepsilon}{2}}q^{\frac{\varepsilon}{8}}
\prod_{n>0}(1-q^{n})
\left(1+(-1)^{\varepsilon'}zq^{n-\frac{1-\varepsilon}{2}}\right)
\left(1+(-1)^{\varepsilon'}z^{-1}q^{n-\frac{1+\varepsilon}{2}}\right)$$
for $\varepsilon,\varepsilon'\in\{0,1\}$.
We note that
$$\vartheta_{\varepsilon,\varepsilon'}(u;\tau)
=i^{-\varepsilon\varepsilon'}
\sum_{n\in\mathbb{Z}}e^{\pi i\tau(n+\frac{\varepsilon}{2})^{2}
+2\pi i(u+\frac{\varepsilon'}{2})(n+\frac{\varepsilon}{2})}
=z^{\frac{\varepsilon}{2}}q^{\frac{\varepsilon}{8}}
\vartheta_{0,0}\left(u+\frac{\varepsilon}{2}\tau+\frac{\varepsilon'}{2};\tau\right).$$
By abuse of notation, 
we regard $\eta(\tau)$ (resp. $\vartheta_{\varepsilon,\varepsilon'}(u;\tau)$)
as a holomorphic function on $\mathbb{H}$ 
(resp. $\mathbb{C}\times\mathbb{H}$) and also as
a convergent series in $q$ (resp. $z$ and $q$).

\subsection{The $\mathcal{N}=2$ superconformal algebra}

The \emph{Neveu--Schwarz sector of the $\mathcal{N}=2$ 
superconformal algebra} (firstly introduced in \cite{adem76}) is the Lie superalgebra
\begin{equation*}
\mf{ns}_{2}=
\bigoplus_{n\in\mathbb{Z}}\mathbb{C}
L_{n}\oplus
\bigoplus_{n\in\mathbb{Z}}\mathbb{C}
J_{n}\oplus
\bigoplus_{r\in\mathbb{Z}+\frac{1}{2}}\mathbb{C}
G^{+}_{r}\oplus
\bigoplus_{r\in\mathbb{Z}+\frac{1}{2}}\mathbb{C}
G^{-}_{r}\oplus
\mathbb{C}C
\end{equation*}
with the $\mathbb{Z}_{2}$-grading
\begin{equation*}
(\mf{ns}_{2})^{\bar{0}}=
\bigoplus_{n\in\mathbb{Z}}\mathbb{C}
L_{n}\oplus
\bigoplus_{n\in\mathbb{Z}}\mathbb{C}
J_{n}\oplus\mathbb{C}C,\ \ 
(\mf{ns}_{2})^{\bar{1}}=
\bigoplus_{r\in\mathbb{Z}+\frac{1}{2}}\mathbb{C}
G^{+}_{r}\oplus
\bigoplus_{r\in\mathbb{Z}+\frac{1}{2}}\mathbb{C}
G^{-}_{r}
\end{equation*}
and with the following 
(anti-)commutation relations:
\begin{flalign*}
&
[L_{n},L_{m}]=(n-m)L_{n+m}+\frac{1}{12}(n^{3}-n)C\delta_{n+m,0},\\ 
&
[L_{n},J_{m}]=-mJ_{n+m}, \ 
[L_{n},G^{\pm}_{r}]=\left(\frac{n}{2}-r\right)G^{\pm}_{n+r}, \\
&
[J_{n},J_{m}]=\frac{n}{3}C\delta_{n+m,0},\ \ 
[J_{n},G^{\pm}_{r}]=\pm G^{\pm}_{n+r}, \\
&
[G^{+}_{r},G^{-}_{s}]=
2L_{r+s}+(r-s)J_{r+s}+
\frac{1}{3}\left(r^{2}-\frac{1}{4}\right)C\delta_{r+s,0}, \\
&
[G^{+}_{r},G^{+}_{s}]=
[\,G^{-}_{r},G^{-}_{s}]=0,\ \ 
[\mf{ns}_{2},C]=\{0\},
\end{flalign*}
for $n,m\in\mathbb{Z}$ and $r,s\in\mathbb{Z}+\frac{1}{2}$.

\subsection{The $\mathcal{N}=2$ vertex operator superalgebra}

Let $\mf{ns}_{2}=
(\mf{ns}_{2})_{+}\oplus
(\mf{ns}_{2})_{0}\oplus
(\mf{ns}_{2})_{-}$
be the triangular decomposition of $\mf{ns}_{2}$,
where 
\begin{equation*}
\begin{array}{lcl}
(\mf{ns}_{2})_{+} & := &
\displaystyle
\bigoplus_{n>0}\mathbb{C}L_{n}\oplus
\bigoplus_{n>0}\mathbb{C}J_{n}\oplus
\bigoplus_{r>0}\mathbb{C}G^{+}_{r}\oplus
\bigoplus_{r>0}\mathbb{C}G^{-}_{r},\\
(\mf{ns}_{2})_{-} & := &
\displaystyle
\bigoplus_{n<0}\mathbb{C}L_{n}\oplus
\bigoplus_{n<0}\mathbb{C}J_{n}\oplus
\bigoplus_{r<0}\mathbb{C}G^{+}_{r}\oplus
\bigoplus_{r<0}\mathbb{C}G^{-}_{r}
,\\
(\mf{ns}_{2})_{0} & := &
\displaystyle
\mathbb{C}L_{0}\oplus
\mathbb{C}J_{0}
\oplus\mathbb{C}C,
\end{array}
\end{equation*}
and set $(\mf{ns}_{2})_{\geq0}:=
(\mf{ns}_{2})_{+}\oplus(\mf{ns}_{2})_{0}$.

For $(h,j,c)\in\mathbb{C}^{3}$, let $\mathbb{C}_{h,j,c}$
be the $(1|0)$-dimensional $(\mf{ns}_{2})_{\geq0}$-module
defined by $(\mf{ns}_{2})_{+}.1:=\{0\},\ 
L_{0}.1:=h,\ J_{0}.1:=j$ and
$C.1:=c$.
Then the induced module $\mathcal{M}_{h,j,c}:=
\Ind_{(\mf{ns}_{2})_{\geq0}}^{\mf{ns}_{2}}
\mathbb{C}_{h,j,c}$
is called the Verma module of $\mf{ns}_{2}$.
Denote by $\mathcal{L}_{h,j,c}$ 
the simple quotient $\mf{ns}_{2}$-module of $\mathcal{M}_{h,j,c}$.

We write 
$V_{c}=V_{c}(\mf{ns}_{2})\cong \mathcal{M}_{0,0,c}
/\bigl(U(\mf{ns}_{2})G^{+}_{-\frac{1}{2}}\ket{0,0,c}
+U(\mf{ns}_{2})G^{-}_{-\frac{1}{2}}\ket{0,0,c}\bigr)$
 for the universal $\mathcal{N}=2$ VOSA,
where $\ket{0,0,c}$ is the highest weight vector of $\mathcal{M}_{0,0,c}$.
When $c\neq0$, we also write
$L_{c}=L_{c}(\mf{ns}_{2})\cong\mathcal{L}_{0,0,c} $ 
for its simple quotient VOSA.
See \cite[\S 2]{sato2016equivalences} for the details.

\subsection{Classification of simple modules}

Let $1\leq r\leq p-1$ and $0\leq s\leq p'-1$. 
We set
$\mathcal{L}_{r,s;\lambda}:=\mathcal{L}_{\Delta_{r,s}-a\lambda^{2},2a\lambda,c_{p,p'}},$
where $a:=\frac{p'}{p}$ and
$\Delta_{r,s}:=\frac{(ar-s)^{2}-a^{2}}{4a}.$
We also set $\lambda_{r,s}:=\frac{r-1}{2}-\frac{s}{2a}$. Note that
$\Delta_{r,s}=a\lambda_{r,s}(\lambda_{r,s}+1)$.

The following classification is obtained by D. Adamovi\'{c} via the Kazama--Suzuki coset construction.

\begin{theom}[{\cite[Theorem 7.1 and 7.2]{Ad99}}]
The complete representatives 
of the isomorphism classes of simple $\frac{1}{2}\mathbb{Z}_{\geq0}$-gradable 
$L_{c_{p,p'}}$-modules are given as follows:
$$\left\{\mathcal{L}_{r,0;\lambda_{r,0}-\theta}\mid
1\leq r\leq p-1,0\leq\theta\leq r-1,\theta\in\mathbb{Z}\right\}
\sqcup
\{\mathcal{L}_{r,s;\lambda}\mid (r,s)\in K_{p,p'},\lambda\in\mathbb{C}\},$$
where 
$K_{p,p'}:=\{(m,n)\in\mathbb{Z}^{2}\mid 1\leq m\leq p-1,1\leq n\leq p'-1,mp'+np\leq pp'\}$.
\end{theom}

\begin{rem}
The index set $K_{p,p'}$ parameterizes the BPZ minimal series 
of central charge $1-\frac{6(p-p')^{2}}{pp'}$ and we have
 $|K_{p,p'}|=\frac{(p-1)(p'-1)}{2}$.
\end{rem}

\begin{cor}\label{nonC2}
Let $c\in\mathbb{C}\setminus\{0\}$. 
Then the simple VOSA $L_{c}$
is $C_{2}$-cofinite if and only if $c=c_{p,1}$ for some $p\in\mathbb{Z}_{\geq3}$.
\end{cor}

\begin{proof}
First, the `if' part follows from the regularity of the VOSA $L_{c_{p,1}}$ proved 
in \cite[Theorem 8.1]{Ad01}
and the super-analog of \cite[Theorem 3.8]{li1999some}.

Next, we verify the `only if' part.
We may assume that $c=c_{p,p'}$
for some pair of coprime integers $(p,p')\in\mathbb{Z}_{\geq2}\times\mathbb{Z}_{\geq1}$.
In fact, otherwise, the simple quotient $L_{c}$
is isomorphic to the non $C_{2}$-cofinite VOSA $V_{c}$ by
 \cite[Corollary 9.1.5 (ii)]{gorelik2007simplicity}.
Here we assume that $p'\neq1$. 
By the previous theorem and 
\cite[Theorem 1.3]{kac1994vertex},
it follows that 
the Zhu algebra of $L_{c_{p,p'}}$ 
 is infinite-dimensional.
Then the infinite-dimensionality of (the even part of) the $C_{2}$-algebra of $L_{c_{p,p'}}$ 
follows from a slight generalization of \cite[Proposition 3.3]{arakawa2014zhu}
to the super case (see also \cite[Introduction]{adamovic2011structure}).
This completes the proof of the `only if' part.
\end{proof}

In this paper, we introduce the notion of atypical and typical modules as follows.

\begin{dfn}\label{typical_atypical}
We call a simple $L_{c_{p,p'}}$-module
$\mathcal{L}_{r,s;\lambda}$ \emph{typical} if
$(r,s)\in K_{p,p'}$ and
$\lambda\in \mathbb{C}\setminus S_{r,s},$
where $S_{r,s}:=\{\lambda_{r,s},\lambda_{p-r,p'-s}\}+\mathbb{Z}$.
Otherwise, we call $\mathcal{L}_{r,s;\lambda}$ \emph{atypical}.
\end{dfn}

\subsection{Formal characters}
Let $M$ be a weight $\bigl(\mf{ns}_{2},(\mf{ns}_{2})_{0}\bigr)$-module
of central charge $c$, {\it i.e.}
$(\mf{ns}_{2})_{0}$ acts semi-simply on $M$
and $C$ acts as the scalar $c$.
We set
$(q,z,w):=\bigl(e^{L_{0}^{*}},e^{J_{0}^{*}},e^{C^{*}}\bigr),$
where $(L_{0}^{*},J_{0}^{*},C^{*})$ is the dual basis of $(\mf{ns}_{2})_{0}^{*}$
with respect to the basis $(L_{0},J_{0},C)$.
Then the formal characters of $M$ is defined by
\begin{equation*}
\ch^{0,0}(M)
:=q^{-\frac{c}{24}}\sum_{\lambda\in(\mf{ns}_{2})_{0}^{*}}
({\sf dim}\,M_{\lambda})\,e^{\lambda},
\end{equation*}
\begin{equation*}
\ch^{0,1}(M)
:=q^{-\frac{c}{24}}\sum_{\lambda\in(\mf{ns}_{2})_{0}^{*}}
(\sdim M_{\lambda})\,e^{\lambda},
\end{equation*}
where $M_{\lambda}:=\{v\in M\,|\,X.v=\langle\lambda,X\rangle v\text{ for any }
X\in(\mf{ns}_{2})_{0}\}$.
In what follows, we write
$\ch^{0,\varepsilon}(M)(q,z,w):=\ch^{0,\varepsilon}(M).$

The following lemma is easily verified (see \cite[Lemma 5.2]{sato2016equivalences}).
\begin{lem}\label{spchar}
For the spectral flow twisted module
$M^{\theta}$ (see Appendix \ref{spfl} for the definition) and $\varepsilon\in\{0,1\}$,
we have 
$$\ch^{0,\varepsilon}(M^{\theta})(q,z,w)=q^{\frac{\ c\theta^{2}}{6}}
z^{\frac{c\theta}{3}}\ch^{0,\varepsilon}(M)(q,zq^{\theta},w).$$
\end{lem}
Now we introduce the ``half-twisted'' characters as follows.
\begin{dfn}\label{hfspchar}
For $\varepsilon\in\{0,1\}$,
the twisted character of $M$ is deifined by
$$\ch^{1,\varepsilon}(M)(q,z,w):=q^{\frac{c}{24}}
z^{-\frac{c}{6}}\ch^{0,\varepsilon}(M)(q,zq^{-\frac{1}{2}},w).$$
\end{dfn}

\section{Typical modular transformation law}

In this section, we derive the modular transformation formula for
the character functions of typical modules.
Throughout this section, we assume that $p'\geq2$ and $\mathcal{L}_{r,s;\lambda}$ is typical.

\subsection{Character formula for typical modules}

We prove the following character formula 
via the quantum BRST reduction from $\widehat{\mf{sl}}_{2|1}$
of level $k=-1+a$.
See Appendix \ref{qred} for the proof.

\begin{theom}\label{typical}
We have an equality of formal series
$$\ch^{\varepsilon,\varepsilon'}(\mathcal{L}_{r,s;\lambda})(q,z,w)
=(-1)^{\varepsilon\varepsilon'}
q^{-a(\lambda+\frac{\varepsilon}{2})^{2}}z^{2a(\lambda+\frac{\varepsilon}{2})}w^{3(1-2a)}
\frac{\vartheta_{\varepsilon,\varepsilon'}(u;\tau)}{\eta(\tau)^{2}}\chi_{r,s}(\tau),$$
where
$$\chi_{r,s}(\tau):=\frac{1}{\eta(\tau)}
\sum_{n\in\mathbb{Z}}\left(
q^{pp'\bigl(n+\frac{rp'-sp}{2pp'}\bigr)^{2}}
-q^{pp'\bigl(n+\frac{-rp'-sp}{2pp'}\bigr)^{2}}\right).$$
\end{theom}

\begin{proof}
By Theorem \ref{reduction}, we obtain
$$\ch^{0,0}(\mathcal{L}_{r,s;\lambda})(q,z,w)
=q^{-a\lambda^{2}}z^{2a\lambda}w^{3(1-2a)}
\frac{\vartheta_{0,0}(u;\tau)}{\eta(\tau)^{2}}\chi_{r,s}(\tau).$$
By some computations, we also 
obtain the characters for $(\varepsilon,\varepsilon')\neq(0,0)$.
\end{proof}

\begin{rem}
The convergent series
$\chi_{r,s}(\tau)$ is the normalized character of the corresponding BPZ minimal series.
The modular transformation is given as follows:
\begin{align*}
&\chi_{r,s}\left(-\frac{1}{\tau}\right)=\sum_{(r',s')\in K_{p,p'}}S_{(r,s),(r',s')}\chi_{r',s'}(\tau),\\
&\chi_{r,s}(\tau+1)=e^{2\pi i\bigl(\frac{(rp'-sp)^{2}}{4pp'}-\frac{1}{24}\bigr)}\chi_{r,s}(\tau),
\end{align*}
where
$$S_{(r,s),(r',s')}:=\sqrt{\frac{8}{pp'}}(-1)^{(r+s)(r'+s')}
\sin\left(\frac{\pi(p-p') rr'}{p}\right)\sin\left(\frac{\pi(p-p')ss'}{p'}\right).$$
See \cite[Proposition 6.3]{IohKogaRV} for the details.
\end{rem}

The next corollary immediately follows from the previous character formula.
\begin{cor}
For any $\theta\in\mathbb{Z}$, 
the spectral flow twisted module $(\mathcal{L}_{r,s;\lambda})^{\theta}$ 
is isomorphic to another typical module $\mathcal{L}_{r,s;\lambda-\theta}$.
In particular, the set of typical modules is closed under the spectral flow.
\end{cor}

Now we introduce the corresponding character function.

\begin{dfn}\label{typicalfunc} Let $(r,s)\in K_{p,p'}$ and $x\in\mathbb{C}$. Then 
a \emph{typical character function} is defined as
the following holomorphic function on $\mathbb{H}\times\mathbb{C}\times\mathbb{C}$:
$$\mathbf{T}^{\varepsilon,\varepsilon'}_{r,s;x}(\tau,u,t)
:=(-1)^{\varepsilon\varepsilon'}
q^{a\left(x-\frac{i\varepsilon}{2}\right)^{2}}z^{2ia\left(x-\frac{i\varepsilon}{2}\right)}w^{3(1-2a)}
\frac{\vartheta_{\varepsilon,\varepsilon'}(u;\tau)}{\eta(\tau)^{2}}
\chi_{r,s}(\tau),$$
where
$(q,z,w)=(e^{2\pi i\tau},e^{2\pi iu},e^{2\pi it}).$
\end{dfn}

The next lemma follows from Definition \ref{typical_atypical}, Theorem \ref{typical},
 and Definition \ref{typicalfunc}.

\begin{lem}
As a function in $(q,z,w)=(e^{2\pi i\tau},e^{2\pi iu},e^{2\pi it})$,
$$\mathbf{T}^{\varepsilon,\varepsilon'}_{r,s;x}(\tau,u,t)
=\ch^{\varepsilon,\varepsilon'}(\mathcal{L}_{r,s;ix})(q,z,w)$$
holds if and only if $x\in\mathbb{R}+i(\mathbb{R}\setminus S_{r,s})$.
\end{lem}

\begin{rem}\label{infinite}
Since the family
$\{\,\mathbf{T}^{\varepsilon,\varepsilon'}_{r,s;x}\mid 
x\in \mathbb{R}+i(\mathbb{R}\setminus S_{r,s})\,\}$
is linearly independent,
the space spanned by the character functions
of simple $\frac{1}{2}\mathbb{Z}_{\geq0}$-gradable $L_{c_{p,p'}}$-modules
is not finite dimensional.
\end{rem}

\subsection{Modular transformation of typical characters}\label{3.2}

The typical character functions satisfy the following modular transformation formula.

\begin{theom}\label{MTt} For $(r,s)\in K_{p,p'}$ and $x\in\mathbb{C}$,
the following hold:
\begin{equation}\label{SofT}
\mathbf{T}_{r,s;x}^{\varepsilon,\varepsilon'}\left(-\frac{1}{\tau},\frac{u}{\tau},
t-\frac{u^{2}}{6\tau}\right)
=\sum_{(r',s')\in K_{p,p'}}
\int_{\mathbb{R}}\mathrm{d}x'\,S_{(r,s;x),(r',s';x')}^{tt,(\varepsilon,\varepsilon')}
\mathbf{T}_{r',s';x'}^{\varepsilon',\varepsilon}(\tau,u,t),
\end{equation}

\begin{equation}\label{TofT}
\mathbf{T}_{r,s;x}^{\varepsilon,\varepsilon'}(\tau+1,u,t)
=e^{2\pi i\bigl(a\left(x-\frac{i\varepsilon}{2}\right)^{2}
+\frac{(rp'-sp)^{2}}{4pp'}-\frac{1-\varepsilon}{8}\bigr)}
\mathbf{T}_{r,s;x}^{\varepsilon,\varepsilon\varepsilon'+(1-\varepsilon)(1-\varepsilon')}(\tau,u,t),
\end{equation}
where
$$S_{(r,s;x),(r',s';x')}^{tt,(\varepsilon,\varepsilon')}
:=i^{-\varepsilon\varepsilon'}
S_{(r,s),(r',s')}\sqrt{2a}
e^{-4\pi ia\left(x-\frac{i\varepsilon}{2}\right)
\left(x'-\frac{i\varepsilon'}{2}\right)}.$$
\end{theom}

\begin{proof}
The $S$-transformaltion (\ref{SofT}) essentially follows from the Gaussian integral
$$\frac{e^{\frac{2\pi iau^{2}}{\tau}}}{\sqrt{2a}}\frac{\tilde{q}^{ax^{2}}
\tilde{z}^{2iax}}{\sqrt{-i\tau}}
=\int_{\mathbb{R}}dx'\,
q^{a(x')^{2}}
z^{2iax'}e^{-4\pi iaxx'},$$
where 
$(\tilde{q},\tilde{z}):=(e^{-\frac{2\pi i}{\tau}},e^{\frac{2\pi iu}{\tau}})$.
The $T$-transformation (\ref{TofT}) is clear.
\end{proof}

\subsection{Property of typical $S$-data}

The data $S_{(r,s;x),(r',s';x')}^{tt,(\varepsilon,\varepsilon')}$
in the previous subsection has the following properties.

\begin{prp}
Let $\mu, \mu'\in K_{p,p'}$ and $x, x'\in\mathbb{R}$. Then we have
$$S_{(\mu;x),(\mu';x')}^{tt,(\varepsilon,\varepsilon')}
=S_{(\mu';x'),(\mu;x)}^{tt,(\varepsilon',\varepsilon)}$$
and
$$\sum_{\mu''\in K_{p,p'}}
\int_{\mathbb{R}}\mathrm{d}x''\,S_{(\mu;x),(\mu'';x'')}^{tt,(\varepsilon,\varepsilon')}
S_{(\mu'';x''),(\mu';x')}^{tt,(\varepsilon',\varepsilon)}
=(-1)^{\varepsilon\varepsilon'}
\delta_{\mu,\mu'}\delta(-x+i\varepsilon-x'),$$
where $\delta(z)$ is the delta distribution.
\end{prp}

\begin{proof}
The former is obvious.
The latter follows from 
\begin{align*}
&\int_{\mathbb{R}}\mathrm{d}x''\,
\sqrt{2a}e^{-4\pi ia\left(x-\frac{i\varepsilon}{2}\right)
\left(x''-\frac{i\varepsilon'}{2}\right)}
\times\sqrt{2a}e^{-4\pi ia\left(x''-\frac{i\varepsilon'}{2}\right)
\left(x'-\frac{i\varepsilon'}{2}\right)}
\\
&\ \ =2a\int_{\mathbb{R}}e^{2\pi i(-x+i\varepsilon-x')\times2a\left(x''-\frac{i\varepsilon'}{2}\right)}\mathrm{d}x''\\
&\ \ =\delta(-x+i\varepsilon-x')
\end{align*}
and the fact that
\begin{equation}\label{Svir}
\sum_{\mu''\in K_{p,p'}}S_{\mu,\mu''}S_{\mu'',\mu'}
=\delta_{\mu,\mu'}.
\end{equation}
For example, see \cite[(9.10)]{IohKogaRV} for the proof of the equality (\ref{Svir}).
\end{proof}

\section{Atypical modular transformation law}

In this section, we present the modular transformation formula for
the character functions of atypical modules (see Definition \ref{typical_atypical}
for the definition of atypical modules).
Its proof is given in Appendix \ref{proofofamt}.
Throughout this section, we assume that $p'\geq1$.

\subsection{Parameterization of atypical modules}

First we assign a triple of certain integers to each atypical module
 as follows.
 
\begin{lem}
Let $M$ be a simple highest weight $\mf{ns}_{2}$-module.
Then the following are equivalent:
\begin{enumerate}
\item $M$ is an atypical module.
\item There exist $1\leq r\leq p-1$, $0\leq s\leq p'-1$, and $\theta\in\mathbb{Z}$
such that
$M\cong
\mathcal{L}(r,s)^{\theta},$
where
$\mathcal{L}(r,s)
:=\mathcal{L}_{r,s;\lambda_{r,s}}
=\mathcal{L}_{a\lambda_{r,s},
2a\lambda_{r,s},c_{p,p'}}.$
\end{enumerate}
\end{lem}

\begin{proof}
It immediately follows from Corollary \ref{classifi}.
\end{proof}

\begin{rem}\label{KWadm}
Note that we have
$\Omega^{+}_{\lambda_{r,s}-\theta}\bigl(L(r,s)\bigr)\cong\mathcal{L}(r,s)^{\theta},$
where $L(r,s)$ is the Kac-Wakimoto admissible 
$\widehat{\mf{sl}}_{2}$-module of highest weight 
$(k-2\lambda_{r,s})\Lambda_{0}+
2\lambda_{r,s}\Lambda_{1}$. Here $k'=-2+a^{-1}$
and $\Lambda_{i}$ stands for the $i$-th fundamental weight
of $\widehat{\mf{sl}}_{2}$.
\end{rem}

Next we compute the set of equivalence class
of all the modules of the form $\mathcal{L}(r,s)^{\theta}$.
When $p'=1$, we have 
$\mathcal{L}(r,0)^{r}\cong\mathcal{L}(p-r,0)$
and the periodic isomorphism
$\mathcal{L}(r,0)^{p}\cong\mathcal{L}(r,0)$
for any $1\leq r\leq p-1$.
Therefore, as well known, there are only finitely many inequivalent
simple $\frac{1}{2}\mathbb{Z}_{\geq0}$-gradable $L_{c_{p,1}}$-modules.
On the other hand, when $p'\neq1$, we obtain 
by Corollary \ref{classifi} the following identification:

\begin{lem}\label{classifi2}
Suppose that $p'\geq2$.
Then $\mathcal{L}(r,s)^{\theta}$ is isomorphic to $\mathcal{L}(r',s')$
 if and only if one of the following holds:
\begin{enumerate}
\item
$(r,s;\theta)=(r,s;0)$ and $(r',s')=(r,s),$
\item $(r,s;\theta)=(r,0;r)$ and 
$(r',s')=(p-r,p'-1)$,
\item
$(r,s;\theta)=(r,p'-1;-p+r)$ and $(r',s')=(p-r,0).$
\end{enumerate}
\end{lem}

As a consequence, 
if $p'\geq2$, there exist infinitely many inequivalent atypical modules.

\subsection{Character formula for atypical modules}

In this subsection, we compute the character formula for atypical modules
and introduce the corresponding meromorphic functions.

Let $(r,s)$ be a pair of integers such that
$1\leq r\leq p-1$ and $0\leq s\leq p'-1$.
We consider the meromorphic function
$$\Psi_{p,p';r,s}(u;\tau):=\sum_{n\in\mathbb{Z}}
\left(\frac{q^{pp'\bigl(n+\frac{rp'-sp}{2pp'}\bigr)^{2}}}{1-zq^{pn}}
-\frac{q^{pp'\bigl(n+\frac{-rp'-sp}{2pp'}\bigr)^{2}}}{1-zq^{pn-r}}\right)$$
on $\mathbb{C}\times\mathbb{H}$, whose divisor is
$$D_{p,r}:=\left\{(\left.m+pn\tau,\tau),\ (m+(pn+r)\tau,\tau)
\,\right|\tau \in\mathbb{H},\ m,n\in\mathbb{Z}\right\}.$$
Then we define the formal series $\Phi_{p,p';r,s}(z,q)$ 
as the expansion of $\Psi_{p,p';r,s}(u;\tau)$ 
with respect to the two variables 
$(z,q)=(e^{2\pi iu},e^{2\pi i\tau})$ in the region
$$\mathbb{A}:=
\bigl\{(u;\tau)\in\mathbb{C}\times\mathbb{H}\,\big|\,|q|<|z|<1\bigr\}
\subset(\mathbb{C}\times\mathbb{H})\setminus D_{p,r}.$$
That is, the formal series 
$\Phi_{p,p';r,s}\left(z,q\right)$ is given by
$$\Phi_{p,p';r,s}\left(z,q\right)
=\sum_{n,m\geq0}z^{m}\varphi_{p,p';r,s}^{n,m}(q)-\sum_{n,m<0}z^{m}\varphi_{p,p';r,s}^{n,m}(q),$$
where
$$\varphi_{p,p';r,s}^{n,m}(q):=
q^{pp'\bigl(n+\frac{rp'-sp}{2pp'}\bigr)^{2}+pnm}
-
q^{pp'\bigl(n+\frac{(p-r)p'+(p'-s)p}{2pp'}\bigr)^{2}+(pn+p-r)m}.$$

\begin{theom}
For any $\theta\in\mathbb{Z}$, we have an equality of formal series
\begin{align*}
\ch^{\varepsilon,\varepsilon'}\bigl(\mathcal{L}(r,s)^{\theta}\bigr)(q,z,w)
=&\,(-1)^{\varepsilon\varepsilon'}q^{-a(\lambda_{r,s}-\theta+\frac{\varepsilon}{2})^{2}}
z^{2a(\lambda_{r,s}-\theta+\frac{\varepsilon}{2})}w^{3(1-2a)}\\
&\times\frac{\vartheta_{\varepsilon,\varepsilon'}(u;\tau)}{\eta(\tau)^{3}}
\Phi_{p,p';r,s}\left((-1)^{1-\varepsilon'}zq^{\theta+\frac{1-\varepsilon}{2}},
q\right).
\end{align*}
\end{theom}

\begin{proof}
By \cite[Theorem 7.13]{sato2016equivalences}, we have
$$\ch^{0,0}\bigl(\mathcal{L}(r,s)\bigr)(q,z,w)
=q^{-a\lambda_{r,s}^{2}}
z^{2a\lambda_{r,s}}w^{3(1-2a)}
\frac{\vartheta_{0,0}(u;\tau)}{\eta(\tau)^{3}}
\Phi_{p,p';r,s}\left(-zq^{\frac{1}{2}},
q\right).$$
By Lemma \ref{spchar} and some computations, we also have the formula for 
$\theta\neq0$ and $(\varepsilon,\varepsilon')\neq(0,0)$.
\end{proof}

\begin{dfn}\label{Adefinition}
An \emph{atypical character function} is
defined as the following 
meromorphic function on $\mathbb{H}\times\mathbb{C}\times\mathbb{C}$:
\begin{align*}
\mathbf{A}^{\varepsilon,\varepsilon'}_{r,s;\theta}(\tau,u,t):=
&\,(-1)^{\varepsilon\varepsilon'}q^{-a(\lambda_{r,s}-\theta+\frac{\varepsilon}{2})^{2}}
z^{2a(\lambda_{r,s}-\theta+\frac{\varepsilon}{2})}w^{3(1-2a)}\\
&\times\frac{\vartheta_{\varepsilon,\varepsilon'}(u;\tau)}{\eta(\tau)^{3}}
\Psi_{p,p';r,s}\left(u_{\varepsilon,\varepsilon'}+\theta\tau,\tau\right),
\end{align*}
where 
$u_{\varepsilon,\varepsilon'}:=u+\frac{1-\varepsilon}{2}\tau+\frac{1-\varepsilon'}{2}$.
\end{dfn}

The next lemma holds by definition.

\begin{lem}
As a function in $(q,z,w)=(e^{2\pi i\tau},e^{2\pi iu},e^{2\pi it})$. 
$$\mathbf{A}^{\varepsilon,\varepsilon'}_{r,s;\theta}(\tau,u,t)
=\ch^{\varepsilon,\varepsilon'}\bigl(\mathcal{L}(r,s)^{\theta}\bigr)(q,z,w)$$
holds if $|q|<|z|\cdot|q|^{\theta+\frac{1-\varepsilon}{2}}<1.$

\end{lem}

\begin{ex} In the trivial case, {\it i.e.} $(p,p';r,s)=(2,1;1,0)$, we have
$$\frac{\eta(\tau)^{3}}
{\vartheta_{\varepsilon,\varepsilon'}(u;\tau)}=
(-1)^{\varepsilon\varepsilon'}q^{-\frac{1}{2}(-\theta+\frac{\varepsilon}{2})^{2}}
z^{-\theta+\frac{\varepsilon}{2}}
\Psi_{2,1;1,0}\left(u_{\varepsilon,\varepsilon'}+\theta\tau;\tau\right).$$
In particular, when $\theta=0$ and $(\varepsilon,\varepsilon')=(1,1)$, it reproduces the 
Ramond denominator identity for $\widehat{\mf{gl}}_{1|1}$
(see \cite[Example 4.1]{kac1994integrable}):
$$\prod_{i=0}^{\infty}\frac{(1-q^{i})^{2}}{(1-zq^{i-1})(1-z^{-1}q^{i})}
=\sum_{n\in\mathbb{Z}}\frac{(-1)^{n}q^{\frac{n(n+1)}{2}}}{1-zq^{n}},$$
where $|q|<|z|<1$.
\end{ex}

\subsection{Discrete spectra}\label{discsp}

In this subsection, we define the index set $\mathscr{S}_{p,p'}$
as a finite subset of $\{(r,s;\theta)\,|\,1\leq r\leq p-1, 0\leq s\leq p'-1,\theta\in\mathbb{Z}\}$.

Let $(k,\ell)$ be the unique pair of integers such that
$p'=kp+\ell$ and $1\leq\ell\leq p-1.$
Since the pair $(p,p')$ is coprime, so is $(p,\ell)$.
Then we obtain the following lemma by easy computations.

\begin{lem}\label{dspectra}
For $0\leq m\leq p-1$, let
$(r_{m},s_{m})$ be the unique pair of integers such that
$0\leq r_{m}\leq p-1$,
$0\leq s_{m}\leq \ell-1$
and
$m=\ell r_{m}-ps_{m}.$
Let $\mathscr{S}_{p,p'}^{(m)}$ be 
the set of solutions $(r,s;\theta)$ of the equation
\begin{equation}\label{equality}
2m+2p'-p=2(r-2\theta)p'-(2s+1)p
\end{equation}
under the condition
$$1\leq r\leq p-1,\ 0\leq s\leq p'-1,\ \theta\in\mathbb{Z}.$$
Then we have
$$\mathscr{S}_{p,p'}^{(m)}=\left\{(r,s;\theta)=(2\theta+1+r_{m},kr_{m}+s_{m};\theta)\left|\,
-\left\lfloor\frac{r_{m}}{2}\right\rfloor\leq\theta
\leq\left\lfloor\frac{p-r_{m}}{2}\right\rfloor-1\right.\right\}.$$
\end{lem}

\begin{dfn}\label{discrete}
We call 
$$\mathscr{S}_{p,p'}:=\bigcup_{m=0}^{p-1}\mathscr{S}_{p,p'}^{(m)}$$
the set of \emph{discrete spectra} of central charge $c=c_{p,p'}$.
\end{dfn}

\begin{lem} We have the following:

\begin{enumerate}
\item $|\mathscr{S}_{p,p'}|=\frac{p(p-1)}{2}.$
\item For $(r,s;\theta),(r',s';\theta')\in\mathscr{S}_{p,p'}$,
the two modules $\mathcal{L}(r,s)^{\theta}$ and $\mathcal{L}(r',s')^{\theta'}$ are isomorphic
 if and only if
$(r,s;\theta)=(r',s';\theta').$
\end{enumerate}
\end{lem}

\begin{proof}
(1) is easily verified by calculation.
(2) directly follows from Lemma \ref{classifi2}. 
\end{proof}

\begin{ex} We present some examples.
\begin{enumerate}
\item
Since $(r_{0},s_{0})=(0,0),$ we have
$$\mathscr{S}_{p,p'}^{(0)}=\left\{
(2i-1,0;i-1)\left|\,1\leq i\leq\left\lfloor\frac{p}{2}\right\rfloor
\right.\right\}.$$

\item If $(p,p')=(p,kp+1)$, we have 
\begin{align*}
\mathscr{S}_{p,kp+1}^{(2u-1)}&=\left\{
\bigl(2i,(2u-1)k;i-u\bigr)\left|\, 1\leq i\leq
\left\lfloor\frac{p-1}{2}\right\rfloor\right.\right\},\\
\mathscr{S}_{p,kp+1}^{(2v)}&=\left\{
(2i-1,2vk;i-1-v)\left|\,1\leq i\leq\left\lfloor\frac{p}{2}\right\rfloor
\right.\right\}
\end{align*}
for $1\leq u\leq \left\lfloor\frac{p}{2}\right\rfloor$ and 
$1\leq v\leq \left\lfloor\frac{p-1}{2}\right\rfloor$.
In particular, when $k=0$, the modules corresponding to $\mathscr{S}_{p,1}$
coincide with the $\mathcal{N}=2$ unitary minimal series 
(see \cite{dorrzapf1998embedding} and \cite[Remark 7.12]{sato2016equivalences}).

\end{enumerate}
\end{ex}

\subsection{Modular transformation of atypical characters}\label{4.4}

For $r,r',s,s',\theta,\theta'\in\mathbb{Z}$, we set
$$S_{(r,s;\theta),(r',s';\theta')}^{aa,(\varepsilon,\varepsilon')}
:=i^{-\varepsilon\varepsilon'}
(-1)^{(1-\varepsilon')s+(1-\varepsilon)s'}
\frac{2}{p}
\sin(\pi arr')e^{\pi ia(r-2\theta-1+\varepsilon)(r'-2\theta'-1+\varepsilon')}.
$$
In addition, if $(r',s')\in K_{p,p'},$ we also set
\begin{flalign*}
S_{(r,s;\theta),(r',s';x)}^{at,(\varepsilon,\varepsilon')}
:=&\ i^{-\varepsilon\varepsilon'}(-1)^{r's+rs'}\frac{2}{p}\sin\bigl(\pi arr'\bigr)\\
&\times
\frac{\sin\left(\frac{\pi ss'}{a}\right)e^{2\pi x}+\sin\bigl(2\pi\lambda_{r',(1+s)s'}\bigr)}
{\cosh(2\pi x)-\cos(2\pi\lambda_{r',s'})}
e^{-4\pi a(\lambda_{r,s}-\theta+\frac{\varepsilon}{2})\left(x-\frac{i\varepsilon'}{2}\right)}
\end{flalign*}
for $x\in\mathbb{R}$.

Now we state the main result of this paper.

\begin{theom}\label{MTa}
The following hold:
\begin{align*}
&\mathbf{A}^{\varepsilon,\varepsilon'}_{r,s;\theta}
\left(-\frac{1}{\tau},\frac{u}{\tau},
t-\frac{u^{2}}{6\tau}\right)\\
&\ =\sum_{(r',s';\theta')\in\mathscr{S}_{p,p'}}
S_{(r,s;\theta),(r',s';\theta')}^{aa,(\varepsilon,\varepsilon')}
\mathbf{A}^{\varepsilon',\varepsilon}_{r',s';\theta'}(\tau,u,t)\\
&\ \ \ \ +\sum_{(r'',s'')\in K_{p,p'}}
\int_{\mathbb{R}}
\mathrm{d}x''\,S_{(r,s;\theta),(r'',s'';x'')}^{at,(\varepsilon,\varepsilon')}
\mathbf{T}^{\varepsilon',\varepsilon}_{r'',s'';x''}(\tau,u,t),\\
&\mathbf{A}^{\varepsilon,\varepsilon'}_{r,s;\theta}(\tau+1,u,t)\\
&\ =e^{2\pi i\left[
(ar-s-p')(\theta+\frac{1-\varepsilon}{2})-a(\theta+\frac{1-\varepsilon}{2})^{2}
-\frac{1-\varepsilon}{8}\right]}
\mathbf{A}^{\varepsilon,
\varepsilon\varepsilon'+(1-\varepsilon)(1-\varepsilon')}_{r,s;\theta}(u;\tau)
\end{align*}
for $1\leq r\leq p-1$, $0\leq s\leq p'-1$, and $\theta\in\mathbb{Z}$.
\end{theom}

See Appendix \ref{proofofamt} for the proof of Theorem \ref{MTa}.

\begin{ex}
When $(p,p')=(3,2)$ and $(r,s;\theta)=(1,0;0)$, we have
\begin{align*}
\mathbf{A}^{\varepsilon,\varepsilon'}_{1,0;0}
\left(-\frac{1}{\tau},\frac{u}{\tau},t-\frac{u^{2}}{6\tau}\right)
=&\ \frac{e^{\frac{\pi i}{6}\varepsilon\varepsilon'}}{\sqrt{3}}
\Big(\mathbf{A}^{\varepsilon',\varepsilon}_{1,0;0}(\tau,u,t)
-e^{\frac{\pi i}{3}\varepsilon}
\mathbf{A}^{\varepsilon',\varepsilon}_{1,1;-1}(\tau,u,t)\\
&-e^{\frac{2\pi i}{3}\varepsilon}
\mathbf{A}^{\varepsilon',\varepsilon}_{2,0;0}(\tau,u,t)
-\int_{\mathbb{R}}
\frac{e^{-\frac{4\varepsilon\pi}{3}y}}
{\cosh(2\pi y)}
\mathbf{T}^{\varepsilon',\varepsilon}_{1,1;y}(\tau,u,t)\mathrm{d}y
\Big).
\end{align*}
\end{ex}

\subsection{Property of atypical $S$-data}

In this subsection, we verify the symmetricity and the unitarity of the data 
$S_{(r,s;\theta),(r',s';\theta')}^{aa,(\varepsilon,\varepsilon')}$.

\begin{lem}\label{calclem}
Let $r,r'\in\mathbb{Z}_{\geq0}$
such that $r,r'\leq p-1$ and $r-r'$ is even.
\begin{enumerate}
\item
We have
$$\sum_{j=0}^{p-1}
\sin\bigl(\pi arj\bigr)\sin\bigl(\pi ar'j\bigr)=\frac{p}{2}\delta_{r,r'}.$$
\item
Suppose that $p$ is odd.
Then we have
$$\sum_{j=1}^{\left\lfloor\frac{p-1}{2}\right\rfloor}
\sin\bigl(\pi ar(2j)\bigr)\sin\bigl(\pi ar'(2j)\bigr)
=\frac{p}{4}\delta_{r,r'}.$$
\end{enumerate}
\end{lem}

\begin{proof}
By using $\sin(x)=\frac{e^{ix}-e^{-ix}}{2i}$ and straightforward calculation,
we can verify both (1) and (2).
\end{proof}

\begin{lem}\label{unitarity}
Let $(r,s;\theta)\in\mathscr{S}^{(m)}_{p,p'}$ and $(r',s';\theta')\in\mathscr{S}^{(m')}_{p,p'}$.
Then we have
$$\sum_{(r'',s'';\theta'')\in\mathscr{S}_{p,p'}}
\sin(\pi arr'')\sin(\pi ar'r'')
e^{\pi ia(r_{m}-r_{m'})
(r''-2\theta'')}=\frac{p^{2}}{4}\delta_{r,r'}\delta_{s,s'}\delta_{\theta,\theta'}.$$
\end{lem}

\begin{proof}
Denote the left hand side by $S$.
It is rewritten as
$$S=
\sum_{m''=0}^{p-1}
e^{\pi ia(r_{m}-r_{m'})(r_{m''}+1)}
\sum_{(r'',s'';\theta'')\in\mathscr{S}^{(m'')}_{p,p'}}
\sin(\pi arr'')\sin(\pi ar'r'').$$
We set
\begin{align*}
&S_{\rm odd}:=\sum_{j=1}^{\left\lfloor\frac{p}{2}\right\rfloor}
\sin\bigl(\pi ar(2j-1)\bigr)\sin\bigl(\pi ar'(2j-1)\bigr),\\
&S_{\rm even}:=\sum_{j=1}^{\left\lfloor\frac{p-1}{2}\right\rfloor}
\sin\bigl(\pi ar(2j)\bigr)\sin\bigl(\pi ar'(2j)\bigr).
\end{align*}
Then, by Lemma \ref{dspectra}, we have
$$\sum_{(r'',s'';\theta'')\in\mathscr{S}^{(m'')}_{p,p'}}
\sin(\pi arr'')\sin(\pi ar'r'')
=
\begin{cases}
S_{\rm odd} & \text{ if $r_{m''}$ is even,}\\
S_{\rm even} & \text{ if $r_{m''}$ is odd.}
\end{cases}$$
By easy computations, we have
$$S_{\rm odd}=
\begin{cases}
\displaystyle
(-1)^{p'(r_{m}-r_{m'})}S_{\rm odd}&
\text{ if }p\text{ is even,}\\
\displaystyle
(-1)^{p'(r_{m}-r_{m'})}S_{\rm even}&
\text{ if }p\text{ is odd}
\end{cases}$$
and
$$S_{\rm even}=
\begin{cases}
\displaystyle
(-1)^{p'(r_{m}-r_{m'})}S_{\rm even}&
\text{ if }p\text{ is even,}\\
\displaystyle
(-1)^{p'(r_{m}-r_{m'})}S_{\rm odd}&
\text{ if }p\text{ is odd.}
\end{cases}$$

First, we assume that $p$ is odd. Since
$\{r_{m''}\mid0\leq m''\leq p-1\}=\{0,1,\ldots,p-1\}$,
 we have
$S=
\sum_{\ell=1}^{p}
e^{2\pi ia(r_{m}-r_{m'})\ell}
S_{\rm even}.$
Since $r_{m}=r_{m'}$ holds if and only if $m=m'$,
combining with Lemma \ref{calclem} (2), we obtain
$S=\frac{p^{2}}{4}\delta_{r,r'}\delta_{s,s'}\delta_{\theta,\theta'}.$

Next, we assume that $p$ is even.
Then $p'$ is odd and it follows that $S=0$ holds if $r_{m}-r_{m'}$ is odd.
In what follows, we assume that $r_{m}-r_{m'}$ is even.
Then we have
\begin{align*}
S&=\sum_{r_{m''}\colon\text{odd}}
e^{\pi ia(r_{m}-r_{m'})(r_{m''}+1)}
S_{\rm even}+\sum_{r_{m''}\colon\text{even}}
e^{\pi ia(r_{m}-r_{m'})(r_{m''}+1)}
S_{\rm odd}\\
&=\frac{p}{2}\delta_{m,m'}
\sum_{\ell=1}^{p-1}
\sin\bigl(\pi ar\ell\bigr)\sin\bigl(\pi ar'\ell\bigr)
=\frac{p^{2}}{4}\delta_{r,r'}\delta_{s,s'}\delta_{\theta,\theta'}.
\end{align*}
This completes the proof.
\end{proof}

\begin{prp}
Let $\lambda=(r,s;\theta),\,\lambda'=(r',s';\theta')\in\mathscr{S}_{p,p'}$.
Then we have 
$$S_{\lambda,\lambda'}^{aa,(\varepsilon,\varepsilon')}
=S_{\lambda',\lambda}^{aa,(\varepsilon',\varepsilon)}$$
and 
$$\displaystyle\sum_{\lambda''\in\mathscr{S}_{p,p'}}
S_{\lambda,\lambda''}^{aa,(\varepsilon,\varepsilon')}
\overline{S_{\lambda'',\lambda'}^{aa,(\varepsilon',\varepsilon)}}
=\delta_{\lambda,\lambda'}.$$
\end{prp}

\begin{proof}
Since the former equality is obvious, we verify the latter one.
The left hand side is equal to
$$e^{2\pi ia(1-\varepsilon')\bigr(\lambda_{r,s}-\theta-(\lambda_{r',s'}-\theta')\bigr)}
\times \frac{4}{p^{2}}\times S.$$
Therefore it immediately follows from the Lemma \ref{unitarity}.
\end{proof}

\section{Conjecture on Verlinde coefficients}

In this section, we consider the structure constants of the ``Verlinde ring''
of the simple VOSA $L_{c_{p,p'}}$ in the spirit of \cite{creutzig2013relating}. 
See \cite{creutzig2013relating}, \cite[\S 5]{alfes2014mock}, and references therein for the details.
Throughout this section, we assume that $p'\geq2$.

Let $\lambda=(r,s;\theta)\in\mathscr{S}_{p,p'}$,
$\mu_{j}=(r_{j},s_{j})\in K_{p,p'}$, and
$x_{j}\in \mathbb{C}\setminus iS_{r_{j},s_{j}}$
for $j\in\{1,2\}.$
Then we define the (atypical-typical-typical) \emph{Verlinde coefficient} by
$$N^{(\mu_{2};x_{2})}_{\lambda,(\mu_{1};x_{1})}:=\sum_{\mu'\in K_{p,p'}}
\int_{\mathbb{R}}\mathrm{d}x'
\frac{S_{\lambda,(\mu';x')}^{at,(0,0)}
S_{(\mu_{1};x_{1}),(\mu';x')}^{tt,(0,0)}
S_{(\mu_{2};-x_{2}),(\mu';x')}^{tt,(0,0)}}{S_{(1,0;0),(\mu';x')}^{at,(0,0)}}.
$$

\begin{conj} The Verlinde coefficient $N^{\mu_{2}}_{\lambda,\mu_{1}}$ 
can be expressed as 
a certain delta function
with a non-negative integer coffiecient.
In addition, the integer coincides with the fusion rule of 
the corresponding simple $L_{c_{p,p'}}$-modules,
{\it i.e.} the dimension of the space of intertwining operators
of type
$$\left(
\begin{array}{c}
\mathcal{L}_{r_{2},s_{2};ix_{2}}\\
\mathcal{L}(r,s)^{\theta}\ \ \mathcal{L}_{r_{1},s_{1};ix_{1}}
\end{array}
\right).$$
\end{conj}

See \cite[Definition 1.6]{kac1994vertex} for the definition of intertwining operators.

\begin{ex}
When $(p,p')=(3,2)$, we have 
$$\mathscr{S}_{3,2}^{(0)}=\left\{\lambda_{0}:=(1,0;0)\right\},\ 
\mathscr{S}_{3,2}^{(1)}=\left\{\lambda_{-1}:=(1,1;-1)
\right\},\ 
\mathscr{S}_{3,2}^{(2)}=\left\{\lambda_{1}:=(2,0;0)\right\}.$$
We also have $K_{3,2}=\{(1,1)\}$ and $S_{1,1}=\{\frac{1}{4}\}+\frac{1}{2}\mathbb{Z}$.
For $j\in\{1,2\}$ and $x_{j}\in \mathbb{C}\setminus iS_{1,1},$
we set $\mu_{j}:=(1,1;x_{j})$.
Then we obtain
$$N^{\mu_{2}}_{\lambda_{\epsilon},\mu_{1}}=\int_{\mathbb{R}}\mathrm{d}x'
e^{2\pi ix'(x_{1}-\frac{i\epsilon}{2}-x_{2})}
=\delta\Bigl(x_{1}-\frac{i\epsilon}{2}-x_{2}\Bigr)$$
for $\epsilon\in\{0,\pm1\}$.
In this case, the conjecture states that
$${\sf dim}\,\mathcal{I}(x_{1},x_{2},\epsilon)=
\begin{cases}
1 & \text{ if }x_{2}=x_{1}-\frac{i\epsilon}{2},\\
0 & \text{ otherwise}
\end{cases}
$$
holds for $(x_{1}, x_{2}, \epsilon)$ as above,
where $\mathcal{I}(x_{1},x_{2},\epsilon)$ is the space of $\mathbb{Z}_{2}$-graded intertwining operators of type
$$\left(
\begin{array}{c}
\mathcal{L}_{-\frac{1}{8}+\frac{2}{3}(x_{2})^{2},\frac{4}{3}ix_{2},-1}\\
\mathcal{L}_{\frac{1}{3}|\epsilon|,\frac{2}{3}\epsilon,-1}\ \ 
\mathcal{L}_{-\frac{1}{8}+\frac{2}{3}(x_{1})^{2},\frac{4}{3}ix_{1},-1}
\end{array}
\right).$$
To the best of our knowledge, 
the fusion rule between the above modules
has not ever appeared in the literature.
\end{ex}

\appendix

\section{Spectral flow automorphisms}\label{spfl}

In this section, we give a brief review of spectral flow automorphisms.

\subsection{Definition}

For $\theta\in\mathbb{Z}$, the \emph{spectral flow automorphism}
 ${\sf U}_{\theta}$ of $\mf{ns}_{2}$ is defined by
\begin{equation*}
\textstyle
{\sf U}_{\theta}(L_{n})
=L_{n}+\theta J_{n}
+\frac{\ \theta^{2}}{6}C\delta_{n,0},\ 
{\sf U}_{\theta}(J_{n})
=J_{n}+\frac{\theta}{3}C\delta_{n,0},\ 
{\sf U}_{\theta}(G^{\pm}_{r})
=G^{\pm}_{r\pm\theta},\ 
{\sf U}_{\theta}(C)=C.
\end{equation*}
Let
${\sf U}_{\theta}^{*}\colon\mf{ns}_{2}\text{\sf-mod}
\longrightarrow \mf{ns}_{2}\text{\sf-mod}$
denote the endofunctor induced  by the pullback action with respect to ${\sf U}_{\theta}$.
For simplicity, we write $M^{\theta}$ for ${\sf U}_{\theta}^{*}(M)$ in this paper.
Since we have 
${\sf U}_{0}=\id_{\mf{ns}_{2}}$ and
${\sf U}_{\theta}\circ {\sf U}_{\theta'}
={\sf U}_{\theta+\theta'}$
for any $\theta,\theta'\in\mathbb{Z}$,
the two functors ${\sf U}_{\theta}^{*}$ and ${\sf U}_{-\theta}^{*}$
are mutually inverse.

Note that we have the following lemma:
\begin{lem}[{\cite[Lemma B.4]{sato2016equivalences}}]
For any $c\in\mathbb{C}$, 
the restriction of ${\sf U}_{\theta}^{*}$ gives the categorical isomorphism
$L_{c}\text{\sf-mod}\to L_{c}\text{\sf-mod};\ M\mapsto M^{\theta}.$
\end{lem}

\subsection{Spectral flow on irreducible highest weight modules}

It is easy to verify the next lemma.

\begin{lem} Let $(h,j,c)\in\mathbb{C}^{3}$. Then we have
$$(\mathcal{L}_{h,j,c})^{1}\cong
\begin{cases}
\mathcal{L}_{h+j+\frac{c}{6},j+\frac{c}{3},c} & \text{ if }h=-\frac{j}{2},\\
\mathcal{L}_{h+j+\frac{c}{6}-\frac{1}{2},j+\frac{c}{3}-1,c} & \text{ if }h\neq-\frac{j}{2}.
\end{cases}
$$
$$(\mathcal{L}_{h,j,c})^{-1}\cong
\begin{cases}
\mathcal{L}_{h-j+\frac{c}{6},j-\frac{c}{3},c} & \text{ if }h=\frac{j}{2},\\
\mathcal{L}_{h-j+\frac{c}{6}-\frac{1}{2},j-\frac{c}{3}+1,c} & \text{ if }h\neq\frac{j}{2}.
\end{cases}
$$
\end{lem}

By the above lemma and some computations, we obtain the following.

\begin{cor}\label{classifi}
Let $1\leq r\leq p-1$, $0\leq s\leq p'-1$, and $\theta\in\mathbb{Z}$.
Then we have
\begin{equation*}
\mathcal{L}(r,s)^{\theta}
\cong 
\begin{cases}
\mathcal{L}_{r,s+1;\lambda_{r,s+1}-\theta}
 & \text{ if }s\neq p'-1 \text{ and } \theta<0\\
\mathcal{L}(p-r,0)
 & \text{ if }s= p'-1 \text{ and } \theta=-p+r\\
\mathcal{L}_{r,s+1;\lambda_{r,s+1}-\theta}
 & \text{ if }s= p'-1 \text{ and } -p+r<\theta<0\\
\mathcal{L}_{r,s;\lambda_{r,s}-\theta}
 & \text{ if }s=0 \text{ and }0\leq\theta<r,\\
\mathcal{L}(p-r,p'-1) 
& \text{ if }s=0 \text{ and }\theta=r,\\
\mathcal{L}_{r,s;\lambda_{r,s}-\theta}
 & \text{ if }s\neq0\text{ and }0\leq\theta,
\end{cases}
\end{equation*}
where $\lambda_{m,n}:=\frac{m-1}{2}-\frac{n}{2a}$
for $m,n\in\mathbb{Z}$.
\end{cor}

\section{Proof of the typical character formula}\label{qred}

\subsection{Affine Lie superalgebra $\widehat{\mf{sl}}_{2|1}$}

Let
$\mf{sl}_{2|1}$
be the simple Lie superalgebra of type $A(1,0)$.
We fix a Cartan subalgebra $\overline{\mf{h}}$ of $\mf{sl}_{2|1}$
and take its $\mathbb{C}$-basis $\bigl\{h_{i}\,\big|\,i\in\{1,2\}\bigr\}$ such that
$$
\begin{pmatrix}
(h_{1}|h_{1}) & (h_{1}|h_{2})\\
(h_{2}|h_{1}) & (h_{2}|h_{2})
\end{pmatrix}
=
\begin{pmatrix}
0 & 1\\
1 & 0
\end{pmatrix},$$
where $(?|?)$ stands for the normalized even supersymmetric invariant bilinear form
 on $\mf{sl}_{2|1}$.
Let $\mf{h}=\overline{\mf{h}}\oplus\mathbb{C}K
\oplus\mathbb{C}D$
denote the Cartan subalgebra of 
the Kac--Moody affinization
$\widehat{\mf{sl}}_{2|1}=\mf{sl}_{2|1}\otimes\mathbb{C}[t,t^{-1}]
\oplus\mathbb{C}K\oplus\mathbb{C}D$.
We define $\beta_{i},\Lambda_{0},\delta\in\mf{h}^{*}$ for $i\in\{1,2\}$ by 
$$\beta_{1}(h_{1})=\beta_{2}(h_{2})=0,\ \beta_{1}(h_{2})=\beta_{2}(h_{1})=1,\ 
\beta_{i}(K)=\beta_{i}(D)=0,$$
$$\Lambda_{0}(h_{i})=\Lambda_{0}(D)=\delta(h_{i})=\delta(K)=0,\ 
\Lambda_{0}(K)=\delta(D)=1.$$
The normalized symmetric invariant bilinear form on $\mf{h}^{*}$ is given by
$$\begin{pmatrix}
(\beta_{1}|\beta_{1}) & (\beta_{1}|\beta_{2}) & (\beta_{1}|\delta) & (\beta_{1}|\Lambda_{0})\\
* & (\beta_{2}|\beta_{2}) & (\beta_{2}|\delta) & (\beta_{2}|\Lambda_{0})\\
* & * & (\delta|\delta) & (\delta|\Lambda_{0})\\
* & * & * & (\Lambda_{0}|\Lambda_{0})
\end{pmatrix}
=\begin{pmatrix}
0 & 1 & 0 & 0\\
* & 0 & 0 & 0\\
* & * & 0 & 1\\
* & * & * & 0
\end{pmatrix}.$$

\subsection{Root system}
For convinience, we put
$\alpha_{1}:=\beta_{1}+\beta_{2}$ and $\alpha_{0}:=\delta-\alpha_{1}.$
Then the root set $\Delta$ of $\widehat{\mf{sl}}_{2|1}$ is given by
$$\Delta=
\bigl\{\alpha_{0}+n\delta,\,\alpha_{1}+n\delta,\,\pm\beta_{i}+n\delta,\,m\delta
\,\big|\,n\in\mathbb{Z},\,m\in\mathbb{Z}\setminus\{0\},\,i\in\{1,2\}\bigr\}.$$ 
Throughout this paper, we fix the simple system 
$\Pi:=\{\alpha_{0},\,\beta_{1},\,\beta_{2}\}$
and a corresponding Weyl vector $\rho:=\Lambda_{0}$.

\subsection{Formal character}

We assume that an $\widehat{\mf{sl}}_{2|1}$-module $M$ 
decomposes into finite-dimensional weight spaces
$$M=\bigoplus_{\lambda\in\mf{h}^{*}}M_{\lambda},$$
where $M_{\lambda}:=\{v\in M\,|\,hv=\lambda(h)v\text{ for any }h\in\mf{h}\}.$
Then we define its formal character by
$$\ch(M):=
\sum_{\lambda\in\mf{h}^{*}}
\dim(M_{\lambda})e^{\lambda}.$$


\subsection{Typical $\widehat{\mf{sl}}_{2|1}$-modules}\label{taffinesuper}

In this subsection, we consider the highest weight $\widehat{\mf{sl}}_{2|1}$-modules
of highest weight
\begin{align*}
\Lambda=\Lambda_{r,s;\lambda}
&:=(-1+a)\Lambda_{0}
+a(-\lambda_{r,s}-\lambda)\beta_{1}
+a(-\lambda_{r,s}+\lambda)\beta_{2}
\end{align*}
for $(r,s)\in K_{p,p'}$ and $\lambda\in\mathbb{C}\setminus S_{r,s}$.
Let $L(\Lambda)$ be the irreducible highest weight $\widehat{\mf{sl}}_{2|1}$-module
of highest weight $\Lambda$ and
$W\bigl(L(\Lambda)\bigr)$
the corresponding integral Weyl group defined in 
\cite[\S 7.3.5]{gorelik2015characters}.

\begin{lem}\label{assumption}
Let $y\in W\bigl(L(\Lambda)\bigr)$.
\begin{enumerate}
\item We have $(\Lambda+\rho\,|\,\beta)\neq0$ 
for any $\beta\in\Delta_{\sf iso}:=\{\beta\in\Delta\,|\,(\beta|\beta)=0\}$.
\item The weight $y\circ\Lambda:=y(\Lambda+\rho)-\rho$ coincides with $\Lambda$
if and only if $y={\sf id}$.
\item We have 
$\Lambda-y\circ\Lambda\in\mathbb{Z}_{\geq0}\Pi.$
\end{enumerate}
\end{lem}

\begin{proof}
First we prove (1).
Since $\Delta_{\sf iso}=\{\pm\beta_{i}+n\delta,\,m\delta\}\subset\Delta$,
we may assume that $\beta=\beta_{i}+n\delta$ or $\beta=m\delta$.
Since we have
$$\Lambda+\rho=a\bigl(\Lambda_{0}
+(-\lambda_{r,s}-\lambda)\beta_{1}
+(-\lambda_{r,s}+\lambda)\beta_{2}\bigr)$$ and
$-\lambda_{r,s}\pm\lambda\notin\mathbb{Z}$,
we conclude $(\Lambda+\rho\,|\,\beta)\neq0$ in each case.

Next we prove (2).
Let $\Delta\bigl(L(\Lambda)\bigr)$
be the (non-isotoropic) integral root system associated with $L(\Lambda)$
(see \cite[\S7.2.1, 7.3.5]{gorelik2015characters} for the definition).
Then, by direct calculation, we can verify that
$$\Delta\bigl(L(\Lambda)\bigr)=\{\alpha_{0}+(p\ell-r)\delta,
\ \alpha_{1}+(p\ell+r-1)\delta\,|\,\ell\in\mathbb{Z}\}$$
and the corresponding reflections are given by
$$W\bigl(L(\Lambda)\bigr)=\{x_{2p\ell-2r+1},\,x_{2p\ell}\,|\,\ell\in\mathbb{Z}\},$$
where $x_{0}:={\sf id}$ and $x_{k}$ for $k\neq0$ 
is defined as the composition of the following $|k|$ reflections:
$$x_{k}:=\begin{cases}
\cdots r_{\alpha_{0}}r_{\alpha_{1}}r_{\alpha_{0}} &
\text{ if }k>0,\\
\cdots r_{\alpha_{1}}r_{\alpha_{0}}r_{\alpha_{1}} &
\text{ if }k<0.
\end{cases}
$$
By direct computations, we obtain 
\begin{align}\label{weyl1}
\Lambda-x_{2p\ell-2r+1}\circ\Lambda
&=(p\ell-r)(p'\ell-s)\delta+(p'\ell-s)\alpha_{0}
\end{align}
and
\begin{align}\label{weyl2}
\Lambda-x_{2p\ell}\circ\Lambda
&=(pp'\ell+rp'-sp)\ell\delta-p'\ell\alpha_{0}
\end{align}
for any $\ell\in\mathbb{Z}$.
This proves (2).

Finially, (3) follows from the above formulae and the assumption $(r,s)\in K_{p,p'}$.
\end{proof}

\begin{theom}\label{TCF}
By setting $q=e^{-\delta}$, $z_{i}=e^{\beta_{i}}$, and $y=e^{\alpha_{0}}$,
we obtain
\begin{align*}
\widetilde{R}\ch\bigl(L(\Lambda)\bigr)
&=e^{\Lambda}\sum_{\ell\in\mathbb{Z}}\left\{
q^{(pp'\ell+rp'-sp)\ell}
y^{p'\ell}
-q^{(p\ell-r)(p'\ell-s)}
y^{-p'\ell+s}
\right\},
\end{align*}
where
$$\widetilde{R}:=\prod_{n=1}^{\infty}\frac{(1-q^{n})^{2}(1-q^{n}z_{1}z_{2})(1-q^{n-1}z_{1}^{-1}z_{2}^{-1})}
{(1+q^{n}z_{1})(1+q^{n-1}z_{1}^{-1})(1+q^{n}z_{2})(1+q^{n-1}z_{2}^{-1})}
$$
is the $\widehat{\mf{sl}}_{2|1}$-denominator in the Ramond sector.
\end{theom}

\begin{proof}
Since $\Lambda$ is non-critical ({\it i.e.,} $\Lambda(K)\neq-1$),
all the assumptions for \cite[Theorem 11.1.1]{gorelik2015characters}
are clarified by Lemma \ref{assumption}.
Therefore we obtain
$$\widetilde{R}\ch\bigl(L(\Lambda)\bigr)
=\sum_{y\in W(L(\Lambda))}{\sf sgn}(y)\ch\bigl(M(y\circ\Lambda)\bigr).$$
Then the required formula follows from (\ref{weyl1}) and (\ref{weyl2}).
\end{proof}

\subsection{Quantum BRST reduction}

Let ${\sf H}^{0}(?)$ denote the $0$-th BRST cohomology functor 
with respect to the principal nilpotent element $f\in\mf{sl}_{2|1}$
and the semisimple element $x:=\frac{1}{2}(h_{1}+h_{2})\in\overline{\mf{h}}$.
See \cite[\S 3]{arakawa2005representation} for the details.

\begin{lem}
The $\mf{ns}_{2}$-module 
${\sf H}^{0}\bigl(L(\Lambda_{r,s;\lambda})\bigr)$
is isomorphic to 
$\mathcal{L}_{r,s;\lambda}.$
\end{lem}

\begin{proof}
Since $\Lambda_{r,s;\lambda}(K-h_{1}-h_{2})
=ar-s-1\notin\mathbb{Z}_{\geq0}$,
the statement in \cite[Theorem 6.7.4]{arakawa2005representation} ensures that
${\sf H}^{0}\bigl(L(\Lambda_{r,s;\lambda})\bigr)$ is an irreducible highest 
weight $\mf{ns}_{2}$-module.
The corresponding highest weight 
is given by
$$(h,j,c)=\bigl(h_{\Lambda}-\Lambda(x+D),\,
\Lambda(h_{1}-h_{2}),\,-3(2\Lambda(K)+1)\bigr),$$
where $h_{\Lambda}=a(\lambda_{r,s}^{2}-\lambda^{2})$ 
is the lowest conformal weight of $L(\Lambda)$
(see \cite[(9.1--3)]{kac2014representations} for the details).
By some computations, the highest weight coincides with
$$(h,j,c)=\bigl(\Delta_{r,s}-a\lambda^{2},2a\lambda,3(1-2a)\bigr).$$
\end{proof}

Since the next character formula is a special case of \cite[Theorem 3.1]{kac2003quantum}
(see also \cite[(9.4)]{kac2014representations}), we omit the proof.

\begin{theom}\label{reduction}
We have
$$R\ch^{0,0}(\mathcal{L}_{r,s;\lambda})
=q^{h_{\Lambda}-\frac{1-2a}{8}}w^{3(1-2a)}
\widetilde{R}\ch\bigl(L(\Lambda)\bigr)
\Big|_{e^{\Lambda_{0}}=1,z_{1}=z^{-1}q^{-\frac{1}{2}},z_{2}=zq^{-\frac{1}{2}}},$$
where
$$R:=
\prod_{n=1}^{\infty}\frac
{(1-q^{n})^{2}}
{(1+zq^{n-\frac{1}{2}})
(1+z^{-1}q^{n-\frac{1}{2}})}
$$
is the $\widehat{\mf{gl}}_{1|1}$-denominator
in the Neveu--Schwarz sector.
\end{theom}

\section{Proof of the atypical modular transformation law}\label{proofofamt}

In this section, we compute the modular ($S$-)transformation of the 
atypical character function
step by step.

\subsection{Expression in terms of Appell--Lerch sums}

For a positive integer $K$, the level $K$ Appell--Lerch sum is defined as
$$A_{K}(u,v;\tau):=z^{\frac{K}{2}}\sum_{n\in\mathbb{Z}}\frac{(-1)^{Kn}y^{n}
q^{\frac{K}{2}n(n+1)}}{1-zq^{n}}.$$
See \cite{semikhatov2005higher} and \cite{alfes2014mock} 
for the fundamental properties of this function.

We obtain the following by Definition \ref{Adefinition}.
  
\begin{lem}\label{itoAL}
We have
\begin{align*}
&\frac{\eta(\tau)^{3}}{\vartheta_{\varepsilon,\varepsilon'}(u;\tau)}
\mathbf{A}^{\varepsilon,\varepsilon'}_{r,s;\theta}(\tau,u,t)\\
&=(-1)^{\varepsilon\varepsilon'+(1-\varepsilon')p'}
q^{(ar-s-p')(\theta+\frac{1-\varepsilon}{2})-a(\theta+\frac{1-\varepsilon}{2})^{2}}
z^{(ar-s-p')-2a(\theta+\frac{1-\varepsilon}{2})}\\
&\ \ \ \times\Big(
A_{2p'}\bigl(u_{\varepsilon,\varepsilon'}+\theta\tau,
(rp'-sp-pp')\tau;p\tau\bigr)\\
&\ \ \ \ \ \ \ \ -q^{r(s+p')}A_{2p'}\bigl(u_{\varepsilon,\varepsilon'}+(\theta-r)\tau,
(-rp'-sp-pp')\tau;
p\tau\bigr)\Big).
\end{align*}
\end{lem}

\subsection{Modular transformation law of Appell--Lerch sum}

The following modular $S$-transformation property 
of the Appell--Lerch sum is a spcecial case of
\cite[Corollary 3.4]{alfes2014mock}.

\begin{prp}\label{SofAppell}
\begin{align*}
A_{2p'}\left(\frac{u}{\tau},\frac{v}{\tau};-\frac{p}{\tau}\right)
=&\ \frac{\tau}{p}e^{-2\pi i(p'u^{2}-uv)/p\tau}
\Bigg(A_{2p'}\Bigl(\frac{u}{p},\frac{v}{p};\frac{\tau}{p}\Bigr)\\
&+\frac{e^{2\pi i\left(a-\frac{1}{2p}\right)u}}{4p'}
\sum_{m=0}^{2p'-1}\vartheta_{1,1}\Bigl(f_{m}(\tau,v);\frac{\tau}{2pp'}\Bigr)
h\Bigl(\frac{u}{p}-f_{m}(\tau,v);\frac{\tau}{2pp'}\Bigr)
\Bigg),
\end{align*}
where
$f_{m}(\tau,v):=\frac{(2p'-1)\tau+2v-2pm}{4pp'}$
and
$h(u;\tau):=\int_{\mathbb{R}}
\frac{e^{\pi i\tau x^{2}-2\pi ux}}{\cosh\pi x}dx$
is the Mordell integral.
\end{prp}

By Lemma \ref{itoAL} and Proposition \ref{SofAppell}, 
we see that the $S$-tansformed character decomposes 
into the following two factors:
$$\mathbf{A}^{\varepsilon,\varepsilon'}_{r,s;\theta}
\left(-\frac{1}{\tau},\frac{u}{\tau},t-\frac{u^{2}}{6\tau}\right)
=(\text{discrete part})+(\text{continuous part}).$$

\subsection{Computation of the discrete part}

In this subsection, we compute the discrete part
by rewriting all the objects in terms of the level $1$
Appell--Lerch sum.

\subsubsection{$S$-transformed side}

In what follows, we write 
$$\tilde{u}_{\varepsilon,\varepsilon'}
:=u+\frac{1-\varepsilon'}{2}\tau-\frac{1-\varepsilon}{2}
\equiv u_{\varepsilon',\varepsilon}\ (\text{mod }\mathbb{Z}).$$
By Lemma \ref{itoAL} and Proposition \ref{SofAppell}, the discrete part 
is given by
\begin{equation}\label{trans1}
A_{2p'}\left(\frac{\tilde{u}_{\varepsilon,\varepsilon'}-\theta}{p},\frac{-rp'}{p};\frac{\tau}{p}\right)
-A_{2p'}\left(\frac{\tilde{u}_{\varepsilon,\varepsilon'}-\theta+r}{p},\frac{rp'}{p};\frac{\tau}{p}\right).
\end{equation}
We can rewrite this in terms of $A_{1}$ by the following lemma:

\begin{lem}\label{trans2}
\begin{align*}
A_{2p'}\left(\frac{u}{p},\frac{v}{p};\frac{\tau}{p}\right)
=&\ \frac{1}{2p'}\sum_{n,m=0}^{p-1}
z^{\frac{2m+2p'-p}{2p}}y^{\frac{n}{p}}q^{\frac{n(p'n+m+p')}{p}-\frac{n}{2}}\\
&\times\sum_{\ell=0}^{2p'-1}
A_{1}\Bigl(u+n\tau,\frac{2p'n+m+p'}{2p'}\tau-\frac{p\tau}{4p'}+\frac{v+\ell}{2p'};\frac{p\tau}{2p'}\Bigr).
\end{align*}
\end{lem}

The proof is straightforward and we omit it.
\subsubsection{Untransformed side}

By \cite[Proposition 3.3]{alfes2014mock}, we have
\begin{equation}\label{levelreduction}
A_{2p'}(u,v;p\tau)=\frac{e^{2\pi i\left(p'-\frac{1}{2}\right)u}}{2p'}
\sum_{\ell=0}^{2p'-1}A_{1}\left(u,\frac{v}{2p'}+\frac{(2p'-1)p}{4p'}\tau
+\frac{\ell}{2p'};\frac{p\tau}{2p'}\right).
\end{equation}
Then the (untrasformed) atypical character is also rewritten in terms of 
$A_{1}$ as follows:

\begin{cor}\label{untrans}
Let $N,M\in\mathbb{Z}$. We have
\begin{align*}
&\frac{\eta(\tau)^{3}}{\vartheta_{\varepsilon,\varepsilon'}(u;\tau)}
\mathbf{A}^{\varepsilon,\varepsilon'}_{r,s;\theta}(\tau,u,t)\\
&=(-1)^{\varepsilon\varepsilon'}\frac{(-i)^{1-\varepsilon'}}{2p'}
q^{\frac{(\theta+Np)\bigl(\spadesuit+2p'(\theta+Np)\bigr)}{2p}
+\frac{(1-\varepsilon)\bigl(\spadesuit-p'\bigr)}{4p}}
z^{\frac{\spadesuit}{2p}-(1-\varepsilon)a}w^{3(1-2a)}\\
&\ \ \times\sum_{\ell=0}^{2p'-1}\Bigg[
A_{1}\left(u_{\varepsilon,\varepsilon'}+(\theta+Np)\tau,
v_{+}(N);\frac{p\tau}{2p'}\right)\\
&\ \ -q^{(Mp-r)(2Np'+Mp'-s-\frac{1}{2})}
A_{1}\left(u_{\varepsilon,\varepsilon'}+\bigl(\theta-r+(N+M)p\bigr)\tau,
v_{-}(N+M);\frac{p\tau}{2p'}\right)\Bigg],
\end{align*}
where $\spadesuit:=2(r-2\theta)p'-(2s+1)p$ and 
$$v_{\pm}(n):=\frac{2npp'\pm rp'-sp}{2p'}\tau-\frac{p\tau}{4p'}+\frac{\ell}{2p'}$$
for $n\in\mathbb{Z}$.
\end{cor}

\begin{proof}
Since we have $A_{1}(u+n\tau,v+n\tau;\tau)
=(-y)^{-n}q^{-\frac{\ n^{2}}{2}}A_{1}(u,v;\tau),$
we only have to verify the equiality for $N=M=0$.
It immediately follows from (\ref{levelreduction}).
\end{proof}

\subsubsection{Conclusion}

\begin{prp} We have
\begin{align*}
(\text{discrete part})
=&\ i^{-\varepsilon\varepsilon'}\frac{2}{p}
\sum_{(r',s';\theta')\in \mathscr{S}_{p,p'}}
(-1)^{(1-\varepsilon')s+(1-\varepsilon)s'}\\
&\times 
\sin(\pi arr')e^{\pi ia(r-2\theta-1+\varepsilon)(r'-2\theta'-1+\varepsilon')}
\mathbf{A}^{\varepsilon',\varepsilon}_{r',s';\theta'}(u;\tau).
\end{align*}

\end{prp}

\begin{proof}
The left hand side is equal to (\ref{trans1}) and can be written in terms of $A_{1}$
by Lemma \ref{trans2}.
The right hand side is also written in terms of $A_{1}$
by Corollary \ref{untrans}.
Then, by choosing appropriate $N$ and $M$, we obtain the formula.
\end{proof}

\subsection{Computation of the continuous part}
In this subsection, we compute the continuous part.

\subsubsection{Computation of the theta part}
First we compute the theta functions in the continuous part. 
The following lemma is easily verified and we omit the proof.

\begin{lem} We have
$$\vartheta_{1,1}\left(u;\frac{\tau}{2pp'}\right)
=-i\sum_{\ell=0}^{2pp'-1}
e^{2\pi i\left(\ell+\frac{1}{2}\right)
\left(u+\frac{1}{2}\right)}q^{\frac{\left(\ell+\frac{1}{2}\right)^{2}}{4pp'}}
\vartheta_{0,0}\left(2pp'u+\Bigr(\ell+\frac{1}{2}\Bigr)\tau;2pp'\tau\right).$$
\end{lem}

\begin{cor}
For $0\leq m\leq 2p'-1$, we have
\begin{flalign*}
(\text{theta})_{m}:=&\ 
\vartheta_{1,1}\left(\Bigl(p'-\frac{1}{2}\Bigr)\frac{\tau}{2pp'}
+\frac{-rp'+sp-pm+pp'}{2pp'};\frac{\tau}{2pp'}\right)\\
&-e^{2\pi ir(a-\frac{1}{2p})}
\vartheta_{1,1}\left(\Bigl(p'-\frac{1}{2}\Bigr)
\frac{\tau}{2pp'}+\frac{rp'+sp-pm+pp'}{2pp'};\frac{\tau}{2pp'}\right)\\
=&\ 2e^{\pi i(ar-s+m)}e^{\pi i\left(\frac{-rp'+sp-pm}{2pp'}\right)}
q^{-\frac{(2p'-1)^{2}}{16pp'}}\\
&\times\sum_{L=p'}^{2pp'+p'-1}
e^{\frac{\pi iL(s-m)}{p'}}
\sin\left(\frac{\pi rL}{p}\right)
\sum_{n\in\mathbb{Z}}q^{pp'\left(n+\frac{L}{2pp'}\right)^{2}}.
\end{flalign*}
\end{cor}

\begin{proof}
By the previous lemma, we have
\begin{align*}
&\vartheta_{1,1}\left(\Bigl(p'-\frac{1}{2}\Bigr)\frac{\tau}{2pp'}
+\frac{\mp rp'+sp-pm+pp'}{2pp'};\frac{\tau}{2pp'}\right)\\
&=ie^{\pi i\left(\frac{\mp rp'+sp-pm}{2pp'}\right)}q^{-\frac{(2p'-1)^{2}}{16pp'}}
\sum_{\ell=0}^{2pp'-1}e^{\mp\frac{\pi i\ell r}{p}}
e^{\frac{\pi i\ell(s-m)}{p'}}q^{\frac{(\ell+p')^{2}}{4pp'}}
\vartheta_{0,0}\bigl((\ell+p')\tau;2pp'\tau\bigr).
\end{align*}
Then the required formula follows.
\end{proof}

\begin{rem}\label{unitarytheta}
When $p'=1$, 
$$\sum_{L=1}^{2p}
(-1)^{Lm}
\sin\left(\frac{\pi rL}{p}\right)
\sum_{n\in\mathbb{Z}}q^{p\left(n+\frac{L}{2p}\right)^{2}}=0$$
holds for $0\leq m\leq1$.
In particular, we get $(\text{theta})_{m}=0$.
\end{rem}

\subsubsection{Computation of the Mordell part}
Second we compute the Mordell integral in the continuous part. 

\begin{lem}
For $0\leq m\leq 2p'-1$, we have
\begin{align*}
&(\text{Mordell})_{m}\\
&\ :=h\left(\frac{u-\theta-\frac{1-\varepsilon}{2}}{p}
-\frac{-rp'+sp-pm+pp'}{2pp'}-\Bigl(\varepsilon'p'-\frac{1}{2}\Bigr)\frac{\tau}{2pp'}
;\frac{\tau}{2pp'}\right)\\
&\ =2p'
e^{2\pi i\left(\theta+\frac{1-\varepsilon}{2}\right)
\left(\varepsilon'a-\frac{1}{2p}\right)}
e^{-\pi i\varepsilon'\left(ar-s+m\right)}
e^{-\pi i\left(\frac{-rp'+sp-pm}{2pp'}\right)}\\
&\ \ \ \ \times
q^{\frac{\left(2\varepsilon'p'-1\right)^{2}}{16pp'}}
z^{-\varepsilon'a+\frac{1}{2p}}
\int_{\mathbb{R}-i\left(\frac{\varepsilon'}{2}-\frac{1}{4p'}\right)}
\frac{e^{2\pi\left[
p'-m-2a(\lambda_{r,s}-\theta+\frac{\varepsilon}{2})\right]y}}
{\sinh\left(2p'\pi y\right)}
q^{ay^{2}}
z^{2iay}\mathrm{d}y.
\end{align*}
\end{lem}

\begin{proof}
Since we have
$$h(u-s\tau;\tau)
=q^{\frac{s^{2}}{2}}z^{-s}\int_{\mathbb{R}-is}
\frac{e^{\pi i\tau x^{2}-2\pi ux}}{\cosh\pi\left(x+is\right)}\mathrm{d}x$$
for $s\in\mathbb{R}$, the equality follows from a direct computation.
\end{proof}

\begin{lem}\label{C13}
We have
$$e^{2p'\pi y}\sum_{m=0}^{2p'-1}\left((-1)^{1-\varepsilon'}e^{-\frac{\pi iL}{p'}}e^{-2\pi y}\right)^{m}
=\frac{2\sinh(2p'\pi y)}{1+e^{-2\pi (y+\frac{i\varepsilon'}{2})-\frac{\pi iL}{p'}}}.$$
\end{lem}

\begin{proof}
Since
$$-\frac{i\varepsilon'}{2}+\left(n+\frac{1}{2}-\frac{L}{2p'}\right)i
\notin \mathbb{R}-i\left(\frac{\varepsilon'}{2}-\frac{1}{4p'}\right)$$
for any $n\in\mathbb{Z}$,
the ratio in the left-hand side is not equal to $1$.
\end{proof}

\subsubsection{Conclusion}
Finally we obtain the explicit form of the continuous part as follows.

\begin{prp}
We have
\begin{align*}
(\text{continuous part})
=&\ i^{1-\varepsilon\varepsilon'}
\frac{2}{p}\sum_{(r',s')\in K_{p,p'}}
(-1)^{r's+rs'}\sin\bigl(\pi arr'\bigr)\\
&\times\int_{\mathbb{R}}
\frac{\sin\left(\frac{\pi ss'}{a}\right)e^{2\pi x'}+\sin\bigl(2\pi\lambda_{r',(1+s)s'}\bigr)}
{\cosh(2\pi x')-\cos(2\pi\lambda_{r',s'})}\\
&\times e^{-4\pi a(\lambda_{r,s}-\theta+\frac{\varepsilon}{2})\left(x'-\frac{i\varepsilon'}{2}\right)}
\mathbf{T}^{\varepsilon',\varepsilon}_{r',s';x'}(u;\tau)\mathrm{d}x'.
\end{align*}
\end{prp}

\begin{proof}
By Lemma \ref{itoAL} and Proposition \ref{SofAppell},
 the continuous part is equal to 
$C\times\vartheta_{\varepsilon,\varepsilon'}(u;\tau)\eta(\tau)^{-3},$
where
\begin{align*}
C:=&\ i^{1-\varepsilon\varepsilon'}
(-1)^{\varepsilon\varepsilon'}\frac{1}{p}
e^{-2\pi ia(1-\varepsilon')
\left(\lambda_{r,s}-\theta+\frac{\varepsilon}{2}\right)}\\
&\times q^{-\frac{1-\varepsilon'}{4}a}z^{-(1-\varepsilon')a}
\left[\frac{e^{2\pi i(\tilde{u}_{\varepsilon,\varepsilon'}-\theta)(a-\frac{1}{2p})}
}{4p'}
\sum_{m=0}^{2p'-1}(\text{theta})_{m}(\text{Mordell})_{m}\right].
\end{align*}
By using Lemma \ref{C13} and
$e^{2\pi i(\tilde{u}_{\varepsilon,\varepsilon'}-\theta)(a-\frac{1}{2p})}
=e^{-2\pi i(\theta+\frac{1-\varepsilon}{2})
(a-\frac{1}{2p})}z^{a-\frac{1}{2p}}q^{\frac{1-\varepsilon'}{2}(a-\frac{1}{2p})},$
 we have

\begin{align*}
C&=
i^{1-\varepsilon\varepsilon'}
(-1)^{\varepsilon\varepsilon'}\frac{2}{p}
\sum_{L=p'}^{2pp'+p'-1}
e^{\frac{\pi iLs}{p'}}
\sin\left(\frac{\pi rL}{p}\right)
\sum_{n\in\mathbb{Z}}q^{pp'\left(n+\frac{L}{2pp'}\right)^{2}}\\
&\ \ \times\int_{\mathbb{R}-i\left(\frac{\varepsilon'}{2}-\frac{1}{4p'}\right)}
\frac{e^{-4\pi a
(\lambda_{r,s}-\theta+\frac{\varepsilon}{2})y}}
{1+e^{-2\pi (y+\frac{i\varepsilon'}{2})-\frac{\pi iL}{p'}}}
q^{ay^{2}}
z^{2iay}\mathrm{d}y\\
&=i^{1-\varepsilon\varepsilon'}(-1)^{\varepsilon\varepsilon'}
\frac{2}{p}\sum_{r'=1}^{p-1}\sum_{s'=1}^{p'-1}
(-1)^{r's+rs'}e^{-\frac{\pi iss'}{a}}
\sin\bigl(\pi arr'\bigr)\\
&\ \ \times\int_{\mathbb{R}-i\left(\frac{\varepsilon'}{2}-\frac{1}{4p'}\right)}
\frac{e^{-4\pi a(\lambda_{r,s}-\theta+\frac{\varepsilon}{2})y}
}{1-e^{-2\pi (y+\frac{i\varepsilon'}{2})-2\pi i\lambda_{r',s'}}}
q^{ay^{2}}
z^{2iay}
\eta(\tau)\chi_{r',s'}(\tau)\mathrm{d}y\\
&=i^{-\varepsilon\varepsilon'}
\frac{2}{p}\sum_{(r',s')\in K_{p,p'}}
(-1)^{r's+rs'}\sin\bigl(\pi arr'\bigr)\\
&\ \ \times\int_{\mathbb{R}+\frac{i}{4p'}}
\frac{\sin\left(\frac{\pi ss'}{a}\right)e^{2\pi Y}+\sin\bigl(2\pi\lambda_{r',(1+s)s'}\bigr)}
{\cosh(2\pi Y)-\cos(2\pi\lambda_{r',s'})}
 e^{-4\pi a(\lambda_{r,s}-\theta+\frac{\varepsilon}{2})\left(Y-\frac{i\varepsilon'}{2}\right)}\\
&\ \ \times(-1)^{\varepsilon\varepsilon'} q^{a\left(Y-\frac{i\varepsilon'}{2}\right)^{2}}
z^{2ia\left(Y-\frac{i\varepsilon'}{2}\right)}
\eta(\tau)\chi_{r',s'}(\tau)\mathrm{d}Y.
\end{align*}
In the last equality, we put $Y:=y+\frac{i\varepsilon'}{2}$.
Finally we shift the region of the integration by the residue theorem
 and obtain the required formula.
\end{proof}

\section{Comparison with the results of Kac-Wakimoto}\label{KWmodular}

\subsection{Modular transformation of the minimal unitary characters}

In this subsection, we explain the relation between our result and
the modular transformation properties
for $\mathcal{N}=2$ minimal unitary characters in \cite[Theorem 6.1]{kac1994integrable}.
We write $A^{(p;0)}$ and $A^{(p;\frac{1}{2})}$
for the finite set $A_{p}^{+}$ and $A_{p}^{-}$ defined in \cite[\S 6]{kac1994integrable}.

\begin{lem}
Let $\varepsilon\in\{0,\frac{1}{2}\}$ and $(j,k)\in A^{(p;\varepsilon)}$.
Then we have
$$\ch_{j,k}^{p;\varepsilon}(\tau,u)=\mathbf{A}^{1-2\varepsilon,0}_{j+k,0;k-\varepsilon}(\tau,u,0),\ \ 
\sch_{j,k}^{p;\varepsilon}(\tau,u)=\mathbf{A}^{1-2\varepsilon,1}_{j+k,0;k-\varepsilon}(\tau,u,0),$$
where $\ch_{j,k}^{p;\varepsilon}(\tau,u)$ and $\sch_{j,k}^{p;\varepsilon}(\tau,u)$
are defined in \cite[\S 6]{kac1994integrable}.
\end{lem}

\begin{proof}
Since we have
$$\Psi^{(\varepsilon)+}(\tau,u)=\frac{\eta(\tau)^{3}}{\vartheta_{1-2\varepsilon,0}\left(u;\tau\right)},\ 
\Psi^{(\varepsilon)-}(\tau,u)=\frac{\eta(\tau)^{3}}
{(-1)^{1-2\varepsilon}\vartheta_{1-2\varepsilon,1}(u;\tau)}$$
in \cite{kac1994integrable}, we obtain
$$\ch_{j,k}^{p;\varepsilon}(\tau,u)=
q^{\frac{jk}{p}}z^{\frac{j-k}{p}}
\frac{\vartheta_{1-2\varepsilon,0}\left(u;\tau\right)}{\eta(\tau)^{3}}
\left(\sum_{n,m\geq0}-\sum_{n,m<0}\right)(-1)^{n+m}z^{-n+m}q^{pnm+jn+km},$$
$$\sch_{j,k}^{p;\varepsilon}(\tau,u)=(-1)^{1-2\varepsilon}
q^{\frac{jk}{p}}z^{\frac{j-k}{p}}
\frac{\vartheta_{1-2\varepsilon,1}(u;\tau)}{\eta(\tau)^{3}}
\left(\sum_{n,m\geq0}-\sum_{n,m<0}\right)z^{-n+m}q^{pnm+jn+km}.$$
Our character formula gives
\begin{align*}
&\mathbf{A}^{1-2\varepsilon,\varepsilon'}_{j+k,0;k-\varepsilon}(\tau,u,t)\\
&=\ch^{1-2\varepsilon,\varepsilon'}\bigl(\mathcal{L}(j+k,0)^{k-\varepsilon}\bigr)(q,z,w)\\
&=(-1)^{(1-2\varepsilon)\varepsilon'}
q^{-\frac{(j-k)^{2}}{4p}}
z^{\frac{j-k}{p}}w^{3(1-2a)}
\frac{\vartheta_{1-2\varepsilon,\varepsilon'}(u;\tau)}{\eta(\tau)^{3}}
\Phi_{p,1;j+k,0}\left((-1)^{1-\varepsilon'}zq^{k},q\right)
\end{align*}
for $\varepsilon'\in\{0,1\}$.
We can verify that these formulae coincide by some computations.
\end{proof}

Through the above identification, we obtain \cite[Theorem 6.1]{kac1994integrable}
(see also \cite{dobrev1987characters}, \cite{matsuo1987character}) as a special case of Theorem \ref{MTa} for $p'=1$ (see Remark \ref{unitarytheta}).

\subsection{Non-unitary spectra of Kac-Wakimoto}\label{nonunitaryKW}

In this subsection, we compare our spectra $\mathscr{S}_{p,p'}$
with the set of highest weights considered in \cite[\S 3]{kac2016representations}.
In what follows, we identify the set of triples $\mathscr{S}_{p,p'}$ with that of
the highest weights of the corresponding $\mf{ns}_{2}$-modules.

Let $\{\Lambda_{0},\Lambda_{1},\Lambda_{2}\}$ 
be the set of fundamental weights of the affine Lie 
superalgebra $\widehat{\mf{sl}}_{2|1}$.
We write $\Lambda(s,i)$
for the integral weight $(p'-1-s)\Lambda_{0}+s\Lambda_{i}$,
 where $0\leq s\leq p'-1$ and $i\in\{1,2\}$.
In \cite[\S 3]{kac2016representations}, 
V.~Kac and M.~Wakimoto considered 
a certain family of simple highest weight $\mf{ns}_{2}$-modules
of central charge 
$$c=3(1-2a)=3\Bigl(1-\frac{2p'}{p}\Bigr)$$
associated with the pair
$\bigl(p, \Lambda(s,i)\bigr)$ (see \cite[Lemma 2.1]{kac2016representations}).
The set of the corresponding $\mf{ns}_{2}$-highest weights
$(h,j)=\bigl(h,j,3(1-2a)\bigr)$ is give as follows:
\begin{equation*}
S\left(a;s,i\right)
:=\left\{(h,j)=\Lambda_{k_{1},k_{2}}\mid
k_{1},k_{2}\in\mathbb{Z}_{\geq0}\text{ and }k_{1}+k_{2}\leq p-1\right\},
\end{equation*}
where $$\Lambda_{k_{1},k_{2}}:=
\left(a\Bigl(\bigl(k_{1}+\frac{1}{2}\bigr)
\bigl(k_{2}+\frac{1}{2}\bigr)-\frac{1}{4}\Bigr)-s\bigl(k_{1}+\frac{1}{2}\bigr),
(-1)^{i}\bigl(a(k_{2}-k_{1})-s\bigr)\right).$$
We note that they assume the integer $p$ to be odd.
See \cite[\S 3]{kac2016representations} for the details.

\begin{ex} 
Here we present two examples.
\begin{enumerate}
\item If $(p,p')=(3,2)$, we have
$$S\left(\frac{2}{3};0,i\right)=
\left\{(0,0), \left(\frac{1}{3},\pm\frac{2}{3}\right)\right\}=\mathscr{S}_{3,2}$$
for $i\in\{1,2\}$.
\item
If $(p,p')=(5,2)$, we have
$$S\left(\frac{2}{5};0,i\right)\cap\mathscr{S}_{5,2}
=\left\{(0,0), \left(\frac{1}{5},\pm\frac{2}{5}\right), \left(\frac{2}{5},\pm\frac{4}{5}\right),
\left(\frac{4}{5},0\right),
\left(\frac{7}{5},\pm\frac{2}{5}\right)\right\},$$
$$S\left(\frac{2}{5};1,i\right)\cap\mathscr{S}_{5,2}
=\left\{\left(\frac{1}{10},\frac{(-1)^{i}}{5}\right)\right\},$$
and
$$\mathscr{S}_{5,2}\subsetneq
S\left(\frac{2}{5};0,i\right)\sqcup
S\left(\frac{2}{5};1,1\right)\sqcup
S\left(\frac{2}{5};1,2\right)$$
for $i\in\{1,2\}$.
\end{enumerate}
\end{ex}

\bibliographystyle{alpha}

\bibliography{ref}

\end{document}